\documentclass[11pt,reqno]{amsart}
\usepackage[vmargin=2cm,hmargin=2cm]{geometry}
\usepackage{amssymb, latexsym, amsfonts,amsmath,amsthm, color,etaremune,enumerate}
\usepackage{caption,float,textcomp,units,xfrac,url}
\title{Ranks of Overpartitions: Asymptotics and Inequalities}
\author{Alexandru Ciolan}
\address{Mathematical Institute, University of Cologne, Weyertal 86--90, 50931
	Cologne, Germany}
\email{aciolan@math.uni-koeln.de}
\newtheorem{Thm}{Theorem}
\newtheorem{Con}{Conjecture}

\newtheorem{Lem}{Lemma}

\newtheorem{Cor}{Corollary}

\theoremstyle{remark}

\newtheorem{rem}{Remark}
\newtheorem{exa}{Example}

\newcommand{\cal}{\mathcal}

\newcommand{\bb}{\mathbb}

\DeclareMathOperator{\Res}{Res}

\DeclareMathOperator{\RE}{Re}
\DeclareMathOperator{\IM}{Im}
\makeatletter
\let\@@pmod\pmod
\DeclareRobustCommand{\pmod}{\@ifstar\@pmods\@@pmod}
\def\@pmods#1{\mkern4mu({\operator@font mod}\mkern 6mu#1)}
\makeatother

\newcommand{\N}{\mathbb{N}}
\renewcommand{\O}{\mathcal{O}}
\newcommand{\U}{\mathcal{U}}
\newcommand{\Z}{\mathbb{Z}}

\newcommand{\R}{\mathbb{R}}
\newcommand{\C}{\mathbb{C}}
\newcommand{\V}{\mathcal{V}}
\renewcommand{\l}{\ell}
\newcommand{\ol}{\overline}
\newcommand{\primesum}{\sideset{}{'}}
\numberwithin{equation}{section} 
\allowdisplaybreaks
\begin{document}
\begin{abstract}
In this paper we compute asymptotics for the coefficients of an infinite class of overpartition rank generating functions. Using these results, we show that $ \ol N(a,c,n), $ the number of overpartitions of $ n $ with rank congruent to $ a  $ modulo $ c, $ is equidistributed with respect to $ 0\le a< c, $ for any $ c\ge2, $  as $ n\to\infty $   and, in addition, we prove some inequalities between ranks of overpartitions conjectured by Ji, Zhang and Zhao (2018), and Wei and Zhang (2018) for $ n=6 $ and  $ n=10. $ 
\end{abstract}
\maketitle
\section{Introduction and statement of results}
\subsection{Motivation}
A \textit{partition} of a positive integer $ n $ is a non-increasing sequence of positive integers (called \textit{parts}), usually written as a sum, which add up to $ n. $ The number of partitions of $ n $ is denoted by $ p(n). $ For example, $ p(4)=5, $ as the partitions of $ 4 $ are $ 4, ~ 3+1, ~ 2+2, ~ 2+1+1,~1+1+1+1. $ Extending the definition, we set, by convention, $p(0)=1$ and $ p(n)=0 $ for $ n<0. $ 
\par Among many other famous achievements, Ramanujan \cite{Ram} proved that if $ n\ge0, $ then 
\begin{align*}\label{RamCongs}
p(5n+4)&\equiv  {0 \pmod*5,}\nonumber\\
p(7n+5)&\equiv  {0 \pmod*7,}\\
p(11n+6)&\equiv  {0 \pmod*{11}.}\nonumber
\end{align*} 
\par In order to give a combinatorial proof of these congruences, Dyson \cite{Dyson} introduced the \textit{rank} of a partition, often known also as \textit{Dyson's rank}, which is defined to be the largest part of the partition minus the number of its parts. Dyson conjectured that the partitions of $ 5n+4 $ form 5 groups of equal size when sorted by their ranks modulo $ 5, $ and that the same is true for the partitions of $ 7n+5 $ when working modulo 7, conjecture which was proved by Atkin and Swinnerton-Dyer \cite{ASD}.
\par An \textit{overpartition} of $ n $ is a partition in which the \textit{first} occurrence of a part may be overlined. We denote by $ \overline p(n) $ the number of overpartitions of $ n. $ For example, $ \overline{p}(4)=14, $ as the overpartitions of $ 4 $ are $4,~\overline 4,~3 + 1,~\overline 3+ 1,~3 +\overline 1,~\overline 3 + \overline 1,~2+ 2,~\ol 2 + 2,~2+ 1+1,~\ol2 + 1 + 1,~2+ \ol1 + 1,~\ol2 +\ol 1 + 1,~1+1 + 1 + 1,~\ol 1+1 + 1 + 1.$
Overpartitions are natural combinatorial structures associated with the $q$-binomial
theorem, Heine's transformation or Lebesgue's identity. For an overview and further motivation, the reader is referred to \cite{CL} and \cite{Pak}. 
\par Both the partition and overpartition ranks have been extensively studied. By proving that some generating functions associated to the rank are holomorphic parts of harmonic Maass forms, Bringmann and Ono \cite{BOno} showed that the rank partition function satisfies some other congruences of Ramanujan type. In the same spirit, Bringmann and Lovejoy \cite{BJoy} proved that the overpartition rank generating function is the holomorphic part of a harmonic Maass form of weight 1/2, while Dewar \cite{Dewar} made certain refinements. 
\par It is customary to denote by $ N(m,n) $ the number of partitions of $ n $ with rank $ m $ and by $ N(a,m,n) $ the number of partitions of $ n $ with rank congruent to $ a $ modulo $ m. $ The corresponding quantities for overpartitions,  $ \ol N(m,n) $ and $ \ol N(a,m,n), $ are denoted by an overline.
\par By means of generalized Lambert series, Lovejoy and Osburn \cite{LO} gave formulas for the rank differences $ \ol N(s,\l,n)-\ol N(t,\l,n) $ for $ \l=3$ and $\l= 5,$ while rank differences for $ \l =7 $ were determined by Jennings-Shaffer \cite{CJS}.~Recently, by using $ q $-series manipulations and the  $ 3 $ and $ 5 $-dissection of the overpartition rank generation function, Ji, Zhang and Zhao \cite{JZZ} proved some identities and inequalities for the rank difference generating functions of overpartitions modulo 6 and 10, and conjectured a few others. Some further inequalities were conjectured by Wei and Zhang \cite{WZ}.
\par It is one goal of this paper to prove these conjectures. The other, more general, goal is to compute asymptotics for the ranks of overpartitions and this is what we will start with, the inequalities mentioned above, as well as the asymptotic equidistribution of $ N(a,c,n), $  following then as a consequence. In doing so, we  rely on the Hardy-Ramanujan circle method and the modular transformations for overpartitions established by Bringmann and Lovejoy \cite{BJoy}. While the main ideas are essentially those used by Bringmann \cite{B} in computing asymptotics for partition ranks, several complications arise and some modifications need to be carried out. 
\par The paper is structured as follows. The rest of this section is dedicated to introducing some notation that is needed in the sequel and to formulating our main results.  An outline of the proof of Theorem \ref{MainThm} is given in Section \ref{Strategy}, and its proof in detail, along with that of the equidistribution of $ N(a,c,n), $ is given in Section \ref{Proof of Theorem 1}. In the final section we show how to use Theorem \ref{MainThm} in order to prove the inequalities conjectured by Ji, Zhang and Zhao \cite{JZZ}, and Wei and Zhang \cite{WZ}, which are stated in Theorems \ref{Thm2}--\ref{Thm4} together with some other inequalities. 
   
\subsection{Notation and preliminaries}\label{Notation}
The overpartition generating function (see, e.g., \cite{CL}) is given by 
\begin{equation}\label{GenFunOver}
\overline P(q):=\sum_{n\ge 0}\overline{p}(n)q^n=\frac{\eta(2z)}{\eta^2(z)}=1+2q+4q^2+8q^3+14q^4+\cdots,
\end{equation}
where $$ \eta(z):=q^{\frac1{24}}\prod_{n=1}^{\infty}(1-q^n) $$ denotes, as usual, Dedekind's eta function and $ q=e^{2\pi iz}, $ with $ z\in\C$ and $\IM(z)>0. $ 
If we use the standard $ q $-series notation 
\begin{align*}
(a)_n&:=\prod_{r=0}^{n-1}(1-aq^r),\\
(a,b)_n&:=\prod_{r=0}^{n-1}(1-aq^r)(1-bq^r),
\end{align*}
for $ a,b\in\C $ and $ n\in\N\cup\{\infty\}, $ then we know from \cite{Lovejoy} that
\begin{align}\label{overpartsrank}
	\O(u;q):=\sum_{n=0}^{\infty}\sum_{m=-\infty}^{\infty}\overline{N}(m,n)u^mq^n&=\sum_{n=0}^{\infty}\frac{(-1)_nq^{\frac12n(n+1)}}{(uq,q/u)_n}\nonumber\\
	&=\frac{(-q)_{\infty}}{(q)_{\infty}}\left(1+2\sum_{n\ge1}\frac{(1-u)(1-u^{-1})(-1)^nq^{n^2+n}}{(1-uq^n)(1-u^{-1}q^n)} \right) .
\end{align}
\par If $ 0<a<c $ are coprime positive integers, and if by $ \zeta_n=e^{\frac{2\pi i}{n}} $ we denote the primitive $ n $-th root of unity, we set
\begin{equation}\label{OA}
\cal O \left( \frac ac;q\right) :=\cal O \left(\zeta_c^a;q \right)=1+\sum_{n=1}^{\infty} A\left(\frac ac;n \right)q^n.
\end{equation}
\par Let $ k $ be a positive integer. Set $  \widetilde{k}=0 $ if $k$ is even, and $ \widetilde{k}=1 $ if $ k $ is odd. Moreover, put $ k_1=\frac{k}{(c,k)}, $ $ c_1=\frac c{(c,k)}, $ and let the integer $ 0\le \ell <c_1 $ be defined by the congruence $ \ell\equiv ak_1\pmod*{c_1}. $ If $ \frac bc\in(0,1), $ let 
\[s(b,c):=\begin{cases}
0 & \text{if~}0<\frac bc\le\frac14,\\
1 & \text{if~}\frac14<\frac bc\le\frac{3}{4},\\
2 & \text{if~}\frac{3}{4}<\frac bc<1,
\end{cases}
\quad\text{and}\quad
t(b,c):=\begin{cases}
1 & \text{if~}0<\frac bc<\frac12,\\
3 & \text{if~}\frac12<\frac bc<1.
\end{cases}
\]
\par Throughout we will use, for reasons of space, the shorthand notation $ s=s(b,c) $ and $ t=t(b,c). $ In what follows, $ 0\le h< k $ are coprime integers (in case $ k=1, $ we set $ h=0 $ and this is the only case when $ h=0 $ is allowed), and $ h'\in\mathbb Z $ is defined by the congruence $ hh'\equiv-1\pmod* k. $ Further, let \[\omega_{h,k}:=\exp\bigg( \pi i\sum_{\mu=0}^{k-1}\left( \left( \frac{\mu}{k}\right) \right)\left( \left(\frac{h\mu}{k} \right) \right)   \bigg) \]
be the multiplier occurring in the transformation law of the partition function, where \[((x)):=\begin{cases}
x-\lfloor x\rfloor-\frac12 & \text{if~}x\in\R\setminus\Z,\\
0 & \text{if~}x\in\Z.
\end{cases}\]
\begin{rem}
The sums $$ S(h,k):=\sum_{\mu=0}^{k-1}\left( \left( \frac{\mu}{k}\right) \right)\left( \left(\frac{h\mu}{k} \right) \right)    $$ are known in the literature as \textit{Dedekind sums.} For a nice discussion of their properties and how to compute them for small values of $ h, $   the reader is referred to \cite[p.~62]{Apostol}. 
\end{rem} 
\par We next define several  Kloosterman sums. Here and throughout we write $ \sum'_h $ to denote summation over the integers  $ 0\le h<k $ that are coprime to $ k. $ 
\par If $ c\mid k, $ let
\begin{equation*}
\label{Kloost0}
A_{a,c,k}(n,m):=(-1)^{k_1+1}\tan\left(\frac{\pi a}{c} \right) \primesum\sum_h\frac{\omega^2_{h,k}}{\omega_{h,k/2}}\cdot \cot\left(\frac{\pi ah'}{c} \right) \cdot e^{-\frac{2\pi ih'a^2k_1}{c}}\cdot e^{\frac{2\pi i}{k}(nh+mh')}, 
\end{equation*} and
\begin{equation*}
\label{Kloost1}
B_{a,c,k}(n,m):=-\frac{1}{\sqrt2}\tan\left(\frac{\pi a}{c} \right) \primesum\sum_h\frac{\omega^2_{h,k}}{\omega_{2h,k}}\cdot \frac{1}{\sin\left(\frac{\pi ah'}{c} \right)}\cdot e^{-\frac{2\pi ih'a^2k_1}{c}}\cdot e^{\frac{2\pi i}{k}(nh+mh')}.
\end{equation*} 
\par If $ c\nmid k $ and $ 0<\frac{\ell}{c_1}\le\frac14,$   let 
\begin{equation*}
D_{a,c,k}(n,m):=\frac1{\sqrt2}\tan\left(\frac{\pi a}{c} \right)\primesum\sum_h\frac{\omega^2_{h,k}}{\omega_{2h,k}}\cdot  e^{\frac{2\pi i}{k}(nh+mh')},
\end{equation*} 
and if $ c\nmid k $ and $ \frac34<\frac{\ell}{c_1}<1,$ let 
\begin{equation*}
D_{a,c,k}(n,m):=-\frac1{\sqrt2}\tan\left(\frac{\pi a}{c} \right)\primesum\sum_h\frac{\omega^2_{h,k}}{\omega_{2h,k}}\cdot  e^{\frac{2\pi i}{k}(nh+mh')}.
\end{equation*}

\par To state our results, we need at last the following quantities. The motivation behind their expressions becomes clear if one writes down explicitly the computations done in Section \ref{Proof of Theorem 1}. If $ c\nmid k, $ let
\begin{equation}\label{delta}
\delta_{c,k,r}:=\begin{cases}
 \frac{1}{16}-\frac{ \ell}{2c_1}+\frac{\ell^2}{c_1^2}-r\frac{\ell}{c_1} &\text{{if~}}0<\frac{\ell}{c_1}\le\frac14,\\
0 & \text{if~}\frac14<\frac{\ell}{c_1}\le\frac34,\\
 \frac{1}{16}-\frac{3\ell}{2c_1}+\frac{\ell^2}{c_1^2}+\frac12-r\left(1-\frac{\ell}{c_1} \right)  & \text{if~}\frac34<\frac{\ell}{c_1}<1,
\end{cases}
\end{equation}
%
and
\begin{equation}\label{mackr}
m_{a,c,k,r}:=\begin{cases}
-\frac{1}{2c_1^2}( 2(ak_1-\l)^2+c_1(ak_1-\ell) +2rc_1(ak_1-\ell)) &\text{{if~}}0<\frac{\ell}{c_1}\le\frac14,\\
0 & \text{if~}\frac14<\frac{\ell}{c_1}\le\frac34,\\
-\frac{1}{2c_1^2}( 2(ak_1-\l)^2 +3c_1(ak_1-\ell )-2rc_1(ak_1-\l)-c_1^2(2r-1)) & \text{if~}\frac34<\frac{\ell}{c_1}<1.
\end{cases}
\end{equation}
%
\subsection{Statement of results}We are now in shape to state our main results.
\begin{Thm}\label{MainThm}
If $ 0<a<c $ are coprime positive integers with $ c>2, $ and $ 
\varepsilon>0 $ is arbitrary, then  
\begin{align}\label{thm}
A\left( \frac ac;n\right) &=i\sqrt{\frac 2n}\sum_{\substack{1\le k\le\sqrt n\\c|k,~2\nmid k}}\frac{B_{a,c,k}(-n,0)}{\sqrt k}\cdot\sinh\left( \frac{\pi\sqrt n}{k}\right)\nonumber \\
&\phantom{=~}+2\sqrt{\frac 2n}\sum_{\substack{1\le k\le \sqrt n\\c\nmid k,~2\nmid k,~c_1\ne4\\r\ge0,~ \delta_{c,k,r}>0}}\frac{D_{a,c,k}(-n,m_{a,c,k,r})}{\sqrt k}\cdot \sinh \left( \frac{4\pi\sqrt{\delta_{c,k,r}n}}{k} \right)+O_c(n^{\varepsilon}).
\end{align}
\end{Thm}
\begin{rem}
In computing the sums $ B_{a,c,k} $ and $ D_{a,c,k} $ from Theorem \ref{MainThm}, the integer $ h' $ is assumed to be even, cf. Bringmann and Lovejoy \cite[pp.~14--15]{BJoy}.
\end{rem} While the sums involved in the asymptotic formula of $A\left( \frac ac;n\right)  $ might look a bit cumbersome at first, for small values of $ c $ they can be computed without much effort. We exemplify below the particular instances when $ c=3 $ and $ c=10; $ we will come back to Example \ref{ex6}, in more detail, in Section \ref{Ineqs}. 
\begin{exa}
If $ a=1 $ and $ c=3, $ the second sum in \eqref{thm} does not contribute (as $ \delta_{3,k,r}=0$),      while the main asymptotic contribution from the first sum is given by the term corresponding to $ k=3. $ If $ h=1, $ we have $ h'=2 $ and $ \omega_{1,3}=e^{\frac{\pi i}{6}}, $ and if $ h=3, $ we have $ h'=-2 $ and $ \omega_{2,3}=e^{-\frac{\pi i}{6}}. $ Without difficulty, we see that $ B_{1,3,3}(-n,0)=-2i\sqrt2 $ if $ n\equiv1\pmod*3,  $ and  $ B_{1,3,3}(-n,0)=i\sqrt2 $ if $ n\equiv0,2\pmod*3,  $ from where  
\[A\left(\frac1{3};n \right)\sim\begin{cases}\phantom{-}\dfrac4{\sqrt{3n}}\tan\left( \dfrac{\pi}{3}\right)\sinh \left( \dfrac{\pi\sqrt{n}}{3} \right)&\text{if~}n\equiv1\pmod*3,\\
-\dfrac2{\sqrt{3n}}\tan\left( \dfrac{\pi}{3}\right)\sinh \left( \dfrac{\pi\sqrt{n}}{3} \right)&\text{if~}n\equiv0,2\pmod*3. \end{cases}\]
\end{exa}
\begin{exa}\label{ex6}
If $ a=1 $ and $ c=10, $ the first sum in \eqref{thm} does not contribute, while the main asymptotic contribution from the second sum is given by the term corresponding to $ k=1. $ In this case, $ k_1=\ell=1, $ $ c_1=10 $  and the only positive value of $\delta_{10,1,r} $ is attained for $ r=0.  $ As such, we have $\delta_{10,1,0}=\frac9{400},$ $ m_{1,10,1,0}=0 $ and $ D_{1,10,1}(-n,0)=\frac1{\sqrt2}\tan\left( \frac{\pi}{10}\right),$ hence 
\[A\left(\frac1{10};n \right)\sim\frac2{\sqrt n}\tan\left( \frac{\pi}{10}\right)\sinh \left( \frac{3\pi\sqrt{n}}{5} \right). \]
\end{exa}
On using Theorem \ref{MainThm} together with the identity
\begin{equation}\label{orthogonality}
\sum_{n=0}^{\infty}\overline N(a,c,n)q^n=\frac1c\sum_{n=0}^{\infty}\overline p(n)q^n+\frac1c\sum_{j=1}^{c-1}\zeta_c^{-aj}\cdot \O(\zeta_c^j;q),
\end{equation}
which follows by the orthogonality of roots of unity, and the well-known fact (see, e.g., \cite{HR}) that  
\[\overline p(n)\sim \frac1{8n}e^{\pi\sqrt n}\] as $ n\to \infty, $ we obtain the following consequence.
\begin{Cor}
\label{Corollary2}
If $ c\ge2 $, then for any $ 0\le a\le c-1 $ we have, as $ n\to\infty, $ \[\ol N(a,c,n)\sim \frac{\overline p(n)}c\sim \frac{1}{c}\cdot \frac{e^{\pi\sqrt n}}{8n}.\]
\end{Cor}
\begin{rem}
A similar result for partition ranks was obtained recently by Males \cite{Males}.  
\end{rem}
\begin{rem}
A Rademacher-type convergent series expansion for $ \ol p(n) $ was found by Zuckerman \cite[p. 321, eq. (8.53)]{Zucker}, and is given by
\[\ol p(n)=\frac1{2\pi} \sum_{2\nmid k}\sqrt k \primesum\sum_h \frac{\omega^2_{h,k}}{\omega_{2h,k}}\cdot e^{-\frac{2\pi inh}{k}}\cdot \frac d{dn}\left( \frac1{\sqrt n}\sinh\left( \frac{\pi\sqrt n}{k}\right)  \right).\]
\end{rem}
The following inequalities were conjectured by Ji, Zhang and Zhao \cite[Conjecture 1.6 and Conjecture 1.7]{JZZ}, and Wei and Zhang \cite[Conjecture 5.10]{WZ}.
\begin{Con}[Ji--Zhang--Zhao, 2018]\label{Conj1}~
\begin{enumerate}[\rm(i)]
\item For $ n\ge0 $ and $ 1\le i\le 4, $ we have \[\overline N(0,10,5n+i)+\overline N(1,10,5n+i)\ge \overline N(4,10,5n+i)+\overline N(5,10,5n+i).\] 
\item For $ n\ge0, $ we have
\[\overline N(1,10,n)+\overline N(2,10,n)\ge \overline N(3,10,n)+\overline N(4,10,n).\] 
\end{enumerate}
\end{Con}
\begin{Con}[Wei--Zhang, 2018]\label{Conj2} For $ n\ge11, $ we have 
\begin{equation}\label{Wei1}
\overline N(0,6,3n)\ge \overline N(1,6,3n)=\overline N(3,6,3n)\ge \overline N(2,6,3n),
\end{equation}
\begin{equation}\label{Wei2}
\overline N(0,6,3n+1)\ge \overline N(1,6,3n+1)=\overline N(3,6,3n+1)\ge \overline N(2,6,3n+1),
\end{equation}
\begin{equation}\label{Wei3}
\overline N(1,6,3n+2)\ge \overline N(2,6,3n+2)\ge\overline N(0,6,3n+2)\ge\overline N(3,6,3n+2).
\end{equation}	
\end{Con}
As an application of Theorem \ref{MainThm}, we prove these conjectures and, in fact, a bit more.
\begin{Thm}\label{Thm2}
For $ n\ge0, $ we have
\begin{equation}\label{eq:1234}
\overline N(1,10,n)+\overline N(2,10,n)\ge \overline N(3,10,n)+\overline N(4,10,n),
\end{equation}
\begin{equation}\label{eq:0325}
\overline N(0,10,n)+\overline N(3,10,n)\ge \overline N(2,10,n)+\overline N(5,10,n),
\end{equation}
\begin{equation}\label{eq:0145}
\overline N(0,10,n)+\overline N(1,10,n)\ge \overline N(4,10,n)+\overline N(5,10,n).
\end{equation}
\end{Thm}
\begin{Thm}\label{Thm3}
For $ n\ge0, $ we have
\begin{equation}\label{eq:0123}
\overline N(0,6,n)+\overline N(1,6,n)\ge \overline N(2,6,n)+\overline N(3,6,n),
\end{equation}
\begin{equation}\label{eq:0312}
\overline N(0,6,3n)+\overline N(3,6,3n)\ge \overline N(1,6,3n)+\overline N(2,6,3n),
\end{equation}
\begin{equation}\label{eq:0312prime}
\overline N(0,6,3n+1)+\overline N(3,6,3n+1)\ge \overline N(1,6,3n+1)+\overline N(2,6,3n+1),
\end{equation}
\begin{equation}\label{eq:0312primeprime}
\overline N(0,6,3n+2)+\overline N(3,6,3n+2)\le \overline N(1,6,3n+2)+\overline N(2,6,3n+2),
\end{equation}
\begin{equation}\label{SolvedWei11}
	\overline N(0,3,3n)\ge \overline N(1,3,3n)=\overline N(2,3,3n),
\end{equation}
\begin{equation}\label{SolvedWei22}
	\overline N(0,3,3n+1)\ge \overline N(1,3,3n+1)=\overline N(2,3,3n+1),
\end{equation}
\begin{equation}\label{SolvedWei33}
	\overline N(0,3,3n+2)\le \overline N(1,3,3n+2)=\overline N(2,3,3n+2).
\end{equation}
\end{Thm}
\begin{Thm}\label{Thm4}
	For $ n\ge11, $ we have
	\begin{equation}\label{SolvedWei1}
	\overline N(0,6,3n)\ge \overline N(1,6,3n)\ge \overline N(2,6,3n),
	\end{equation}
	\begin{equation}\label{SolvedWei2}
	\overline N(0,6,3n+1)\ge \overline N(1,6,3n+1)\ge \overline N(2,6,3n+1),
	\end{equation}
	\begin{equation}\label{SolvedWei3}
	\overline N(1,6,3n+2)\ge\overline N(2,6,3n+2)\ge\overline N(0,6,3n+2)\ge\overline N(3,6,3n+2).
	\end{equation}
	
\end{Thm}
%

\begin{rem}Similar identities and inequalities were studied, for instance, by Alwaise, Iannuzzi and Swisher \cite{AIS}, Bringmann \cite{B}, and Mao \cite{Mao} for ranks of partitions, and  by Jennings-Shaffer and Reihill \cite{CJSREIH}, and Mao \cite{Mao2} for $ M_2 $-ranks of partitions without repeated odd parts. By establishing identities for the overpartition rank generating functions evaluated at roots of unity analogous to those found in \cite[pp.~35--38]{CJSREIH} for the $ M_2 $-rank, the reader can come up with many other such inequalities. \end{rem}
\begin{rem}
	Ji, Zhang and Zhao \cite{JZZ} proved \eqref{eq:0145} for $ n\equiv0\pmod*5, $ whereas the inequality \eqref{eq:0123} is new. 
\end{rem}
\begin{rem}\label{remarkmod6} The identities from \eqref{Wei1} and \eqref{Wei2} were proved by Ji, Zhang and Zhao \cite{JZZ}, who further proved that $N(0,6,3n)>\ol N(2,6,3n)$ for $ n\ge1, $ and  $ \ol N(0,6,3n+1)>\ol N(2,6,3n+1) $ for $ n\ge 0. $
While \eqref{eq:0123} follows easily now for $ n\equiv0,1\pmod*3, $ the inequality  is not at all clear for $ n\equiv2\pmod*3 $,  as the same authors also showed that $ \ol N(0,6,3n+2)<\ol N(2,6,3n+2)$ for $ n\ge1 $  and $ \ol N(1,6,3n+2)>\ol N(3,6,3n+2)$ for $ n\ge0. $ For a list of the identities and inequalities already proven, see \cite[Theorem 1.4]{JZZ}. 
\end{rem}
\begin{rem}
The identity and inequalities from \eqref{Wei1} were also proved by Wei and Zhang \cite[p.~25]{WZ}.
\end{rem}


\section{Strategy of the proof}\label{Strategy}
For the reader's benefit, we outline the main steps in proving Theorem \ref{MainThm}, along with several other estimates that will be used in what follows.
\subsection{Circle method}\label{SubsectionCircleMethod} The main idea of our approach is the Hardy-Ramanujan circle method. By Cauchy's Theorem we have, for $ n>0, $
\[A\left(\frac ac;n \right)=\frac1{2\pi i}\int_{\cal C}\frac{\O\left(\frac ac;q \right)}{q^{n+1}}dq,\]
where $\cal C $ may be taken to be the circle of radius $ e^{-\frac{2\pi}{n}} $ parametrized by $ q=e^{-\frac{2\pi}{n}+2\pi it} $ with $ t\in[0,1], $ in which case we obtain
\[A\left(\frac ac;n \right)=\int_0^1\O\left(\frac ac;e^{-\frac{2\pi}{n}+2\pi it} \right)\cdot e^{2\pi-2\pi int}dt.\]
If $ \frac{h_1}{k_1}<\frac hk<\frac{h_2}{k_2} $ are adjacent Farey fractions in the Farey sequence of order $ N:=\lfloor n^{1/2}\rfloor, $ we put \[\vartheta_{h,k}':=\frac1{k(k_1+k)}\quad\text{and}\quad\vartheta_{h,k}'':=\frac1{k(k_2+k)}.\]
Splitting the path of integration along the Farey arcs $ -\vartheta_{h,k}'\le\Phi\le \vartheta_{h,k}'', $ where $ \Phi:=t-\frac hk $ and $ 0\le h< k\le N $ with $(h,k)=1, $ we have
\begin{equation}\label{CircleFarey}
A\left(\frac ac;n \right)=\sum_{h,k}e^{-\frac{2\pi inh}{k}}\int_{-\vartheta_{h,k}'}^{\vartheta_{h,k}''} \cal O\left(\frac ac; e^{\frac{2\pi i}{k}(h+iz)}\right)\cdot e^{\frac{2\pi nz}{k}}  d\Phi,  \end{equation}
where $ z=\frac kn-k\Phi i. $
\par The reader familiar with some basics from Farey theory might immediately recognize the inequality 
\begin{equation*}
\frac1{k+k_j}\le\frac1{N+1}
\end{equation*}
for $ j=1,2, $ together with several other known facts (which are otherwise very easy to prove) such as
\begin{equation*}
{\rm Re}(z)=\frac kn,\quad  {\rm Re}\left( \frac 1z\right)>\frac k2,\quad |z|^{-\frac12}\le n^{\frac12}\cdot k^{-\frac12} \quad\text{and}\quad \vartheta_{h,k}'+\vartheta_{h,k}''\le \frac{2}{k(N+1)}.
\end{equation*}
\par For a nice introduction to Farey fractions, one can consult \cite[Chapter 5.4]{Apo}.
\subsection{Modular transformation laws}
Our next step in the proof of Theorem \ref{MainThm} requires the modular transformations\footnote{In passing, we correct the definitions of $ \U (a,b,c;q) $ and $ \V(a,b,c;q), $ as some  misprints occurred in their original expressions from \cite[p.~8]{BJoy}. The necessary changes become clear on consulting the proof, see \cite[pp.~11--17]{BJoy}.} for $ \O\left(\frac ac;q \right)  $ established by Bringmann and Lovejoy \cite{BJoy}, the proof of which can be found in \cite[pp.~11--17]{BJoy}. For $ 0<a<c $ coprime with $ c>2, $ and $ s=s(b,c) $ and  $ t=t(b,c) $ as in Section \ref{Notation}, let

\begin{align*}\U\left(\frac ac;q \right)=\U\left(\frac ac;z \right)&:=\frac{\eta\left(\frac z2 \right) }{\eta^2(z)}\sin\left(\frac{\pi a}{c} \right)\sum_{n\in\Z}\frac{(1+q^n)q^{n^2+\frac n2}}{1-2\cos\left(\frac{2\pi a}{c} \right)q^n+q^{2n} },   \\
\U(a,b,c;q)=\U(a,b,c;z)&:=\frac{\eta\left(\frac z2 \right) }{\eta^2(z)} e^{\frac{\pi ia}{c}\left(\frac{4b}{c}-1-2s \right)}q^{\frac{sb}{c}+\frac{b}{2c}-\frac{b^2}{c^2}}\sum_{m\in\Z}\frac{q^{\frac{m}{2}(2m+1)+ms}}{1-e^{-\frac{2\pi ia}{c}}q^{m+\frac bc}},\\
\cal V\left(a,b,c; q\right)=\V(a,b,c;z)&:=\frac{\eta\left(\frac z2 \right) }{\eta^2(z)}e^{\frac{\pi i a}{c}\left(\frac{4b}{c}-1-2s \right) }q^{\frac{sb}{c}+\frac{b}{2c}-\frac{b^2}{c^2}}\sum_{m\in\bb Z}\frac{q^{\frac{m(2m+1)}{2}+ms}\left( 1+e^{-\frac{2\pi i a}{c}}q^{m+\frac bc}\right)  }{1-e^{-\frac{2\pi i a}{c}}q^{m+\frac bc}},\\
\cal O\left(a,b,c; q\right)=\O(a,b,c;z)&:=\frac{\eta(2z)}{\eta^2(z)}e^{\frac{\pi i a}{c}\left(\frac{4b}{c}-1-t \right) }q^{\frac{tb}{2c}+\frac b{2c}-\frac{b^2}{c^2}}\sum_{m\in\Z}(-1)^m\frac{q^{\frac{m}{2}(2m+1)+\frac{mt}{2}} }{1-e^{-\frac{2\pi i a}{c}}q^{m+\frac bc}},  \\
\V\left(\frac ac; q\right)=\V\left(\frac ac;z \right)&:= \frac{\eta(2z)}{\eta^2(z)}q^{\frac14}\sum_{m\in\Z}\frac{q^{m^2+m}\left( 1+e^{-\frac{2\pi ia}{c}}q^{m+\frac12}\right) }{1-e^{-\frac{2\pi ia}{c}}q^{m+\frac12}}  .\end{align*}
Furthermore, if 
\begin{equation}\label{Hac}
H_{a,c}(x):=\frac{e^x}{1-2\cos\left(\frac{2\pi a}{c} \right)e^x+e^{2x} },
\end{equation}
we consider, for $\nu\in\mathbb Z,$ $ k\in\bb N $ and $ \widetilde k $ as defined in Section \ref{Notation}, the Mordell-type integral 
\[I_{a,c,k,\nu}:=\int_{\R}e^{-\frac{2\pi zx^2}{k}}H_{a,c}\left( \frac{2\pi i\nu}{k}-\frac{2\pi zx}{k}-\frac{\widetilde{k}\pi i}{2k} \right) dx. \]
If $ k $ is even and $ c\mid k, $ or if $ k $ is odd, $ a=1 $ and $ c=4k, $ there might be a pole at $ x=0. $ In these cases we need to take the \textit{Cauchy principal value} of the integral. We will make this precise at a later stage.
\par By using Poisson summation and proceeding similarly to Andrews \cite{AndrewsThesis}, Bringmann and Lovejoy \cite{BJoy} proved the following transformation laws\footnote{Some further corrections are in order; namely, the ``$-  $" sign in front of the expressions from (3)--(6) in their original formulation \cite[Theorem 2.1]{BJoy} should be a ``$ + $", and the other way around for (1) and (2), the reason being that the ``$ \pm $" sign from the expression of the residues $ \lambda_{n,m}^{\pm} $ (see \cite[p. 13]{BJoy}) is meant to be a ``$ \mp $". All necessary modifications are made here.}.
\begin{Thm}[{\cite[Theorem 2.1]{BJoy}}]\label{Transformations}
Assume the notation above and let $ q=e^{\frac{2\pi i}{k}(h+iz)} $ and $ q_1=e^{\frac{2\pi i}{k}\left( h'+\frac iz\right) }, $ with $ z\in\C  $ and $ \RE (z)>0. $ 
\begin{enumerate}[{\rm(1)}]
\item If $ c\mid k $ and $2\mid k, $  then 
\begin{multline*}\O\left(\frac ac;q \right)=(-1)^{k_1+1}i\cdot e^{-\frac{2\pi a^2h'k_1}{c}}\cdot\tan\left(\frac{\pi a}{c} \right)\cdot\cot\left( \frac{\pi ah'}{c}\right)\frac{\omega^2_{h,k}}{\omega_{h,k/2}}z^{-\frac12}\cdot \O\left(\frac{ah'}{c};q_1 \right) \\ +  \frac{4\sin^2\left(\frac{\pi a}{c}\right)\cdot\omega^2_{h,k}}{\omega_{h,k/2}\cdot k}z^{-\frac12}\sum_{\nu=0}^{k-1}(-1)^{\nu}e^{-\frac{2\pi ih'\nu^2}{k}}\cdot I_{a,c,k,\nu}(z) .\end{multline*}
\item If $ c\mid k $ and $ 2\nmid k, $  then 
\begin{multline*}\O\left(\frac ac;q \right)=-\sqrt 2i\cdot e^{\frac{\pi ih'}{8k}-\frac{2\pi i a^2h'k_1}{c}}\cdot\tan\left(\frac{\pi a}{c} \right)\frac{\omega^2_{h,k}}{\omega_{2h,k}}z^{-\frac12}\cdot \U\left(\frac{ah'}{c};q_1 \right) \\ +  \frac{4\sqrt2\sin^2\left(\frac{\pi a}{c}\right)\cdot\omega^2_{h,k}}{\omega_{2h,k}\cdot k}z^{-\frac12}\sum_{\nu=0}^{k-1}e^{-\frac{\pi ih'}{k}(2\nu^2-\nu)}\cdot I_{a,c,k,\nu}(z) .\end{multline*}
\item If $ c\nmid k, $ $ 2\mid k $  and $ c_1\ne 2, $ then  
\begin{multline*}\O\left(\frac ac;q \right)=2 e^{-\frac{2\pi ia^2h'k_1}{c_1c}}\cdot\tan\left(\frac{\pi a}{c} \right)\frac{\omega^2_{h,k}}{\omega_{h,k/2}}z^{-\frac12}\cdot (-1)^{c_1(\ell+k_1)}\cdot \O\left(ah',\frac{\ell c}{c_1},c;q_1 \right) \\ +  \frac{4\sin^2\left(\frac{\pi a}{c}\right)\cdot\omega^2_{h,k}}{\omega_{h,k/2}\cdot k} z^{\frac12}\sum_{\nu=0}^{k-1}(-1)^{\nu}e^{-\frac{2\pi ih'\nu^2}{k}}\cdot I_{a,c,k,\nu}(z) .\end{multline*}
\item If $ c\nmid k, $ $ 2\mid k $  and $ c_1=2, $ then 
\begin{multline*}\O\left(\frac ac;q \right)= e^{-\frac{\pi ia^2h'k_1}{c}}\cdot\tan\left(\frac{\pi a}{c} \right)\frac{\omega^2_{h,k}}{\omega_{h,k/2}\cdot k}z^{-\frac12}\cdot \V\left(\frac{ah'}{c};q_1 \right) \\ +  \frac{4\sin^2\left(\frac{\pi a}{c}\right)\cdot\omega^2_{h,k}}{\omega_{h,k/2}\cdot k} z^{\frac12}\sum_{\nu=0}^{k-1}(-1)^{\nu}e^{-\frac{2\pi ih'\nu^2}{k}}\cdot I_{a,c,k,\nu}(z) .\end{multline*}
\item If $ c\nmid k, $ $ 2\nmid k $  and $ c_1\ne 4, $ then 
\begin{multline*}\O\left(\frac ac;q \right)=\sqrt 2 e^{\frac{\pi ih'}{8k}-\frac{2\pi i a^2h'k_1}{c_1c}}\cdot\tan\left(\frac{\pi a}{c} \right)\frac{\omega^2_{h,k}}{\omega_{2h,k}}z^{-\frac12}\cdot \U\left(ah',\frac{\ell c}{c_1},c;q_1 \right) \\ +  \frac{4\sqrt2\sin^2\left(\frac{\pi a}{c}\right)\cdot\omega^2_{h,k}}{\omega_{2h,k}\cdot k} z^{\frac12}\sum_{\nu=0}^{k-1}e^{-\frac{\pi ih'}{k}(2\nu^2-\nu)}\cdot I_{a,c,k,\nu}(z) .\end{multline*}
\item If $ c\nmid k, $ $ 2\nmid k $  and $ c_1= 4, $ then 
\begin{multline*}\O\left(\frac ac;q \right)= e^{\frac{\pi ih'}{8k}-\frac{2\pi i a^2h'k_1}{c_1c}}\cdot\tan\left(\frac{\pi a}{c} \right)\frac{\omega^2_{h,k}}{\sqrt2\cdot\omega_{2h,k}}z^{-\frac12}\cdot \V\left(ah',\frac{\ell c}{c_1},c;q_1 \right) \\ +  \frac{4\sqrt2\sin^2\left(\frac{\pi a}{c}\right)\cdot\omega^2_{h,k}}{\omega_{2h,k}\cdot k} z^{\frac12}\sum_{\nu=0}^{k-1}e^{-\frac{\pi ih'}{k}(2\nu^2-\nu)}\cdot I_{a,c,k,\nu}(z) .\end{multline*}
\end{enumerate}
\end{Thm}
In addition to these modular transformations, we need some further estimates.
\subsection{The Mordell integral $ I_{a,c,k,\nu} $}
In the previous subsection we introduced 
\begin{equation}\label{Mordell}
I_{a,c,k,\nu}=\int_{\R}e^{-\frac{2\pi zx^2}{k}}H_{a,c}\left( \frac{2\pi i\nu}{k}-\frac{2\pi zx}{k}-\frac{\widetilde{k}\pi i}{2k} \right) dx.
\end{equation}
Recalling the definition \eqref{Hac}, it is easy to see that 
\begin{equation*}\label{Hacsinh}
H_{a,c}(x)=\frac1{4\sinh\left(\frac x2+ \frac{\pi ia}{c} \right)\sinh\left(\frac x2-\frac{\pi ia}{c} \right) },
\end{equation*}
and so $ H_{a,c}(x) $ can only have poles in points of the form  
\begin{equation*}\label{polesHac}
x=2\pi i\left(n\pm\frac ac \right) 
\end{equation*} with $ n\in\Z. $
\par For $ 2\mid k, $ $ c\mid k $ and $ \nu=\frac {ka}c $ or $ \nu=k\left(1-\frac ac \right),  $  there may be a pole at $ x=0. $ The same is true if $ 2\nmid k ,$ $ \nu=0, $ $ a=1 $ and $ c=4k. $ In both cases we must consider the Cauchy principal value of the integral $ I_{a,c,k,\nu}, $ that is, instead of $ \R $ we choose as path of integration the real line indented below 0.
\par The following\footnote{Note that there are a few typos in the formulation of the original result from which this lemma is inspired. More precisely, in the statement of \cite[Lemma 3.1]{B}, $ n^{\frac14} $ should read $ n^{\frac12}, $ $ k $ should read $ k^{-\frac12} $ and the $ 6kc $ factor from the definition of $ g_{a,c,k,\nu} $ should be removed. These changes, however, do not affect the proof.} is adapted after \cite[Lemma 3.1]{B}.
\begin{Lem}
\label{LemmawithMordell} Let $ n\in \bb N, $ $ N=\lfloor n^{1/2} \rfloor$ and  $ z=\frac kn-k\Phi i,  $ where $ -\frac1{k(k+k_1)}\le\Phi\le\frac1{k(k+k_2)} $ and $ \frac{h_1}{k_1}<\frac hk<\frac{h_2}{k_2} $ are adjacent Farey fractions in the Farey sequence of order $ N. $  If
$$ g_{a,c,k,\nu}:=\begin{cases}\left(\min \left\lbrace \left| \frac{\nu}{k}-\frac1{4k}+\frac ac\right| , \left| \frac{\nu}{k}-\frac1{4k}-\frac ac \right| \right\rbrace  \right) ^{-1} &\text{if~$ k $~is odd, $ \nu\ne0 $ and $ \frac a{c}\ne\frac1{4k} $},\\\left(\min \left\lbrace \left| \frac{\nu}{k}+\frac ac\right| , \left| \frac{\nu}{k}-\frac ac\right| \right\rbrace \right) ^{-1} &\text{if~$ k $~is even and $\nu\not\in \left\lbrace \frac{ka}{c},k\left(1-\frac ac \right)\right\rbrace  $},\\
 \frac ca &\text{otherwise},  \end{cases} $$ and $ \{x\}=x-\lfloor x\rfloor $ is the fractional part of  $ x\in\bb R, $ then
 \[z^{\frac12}\cdot I_{a,c,k,\nu}\ll_c k^{-\frac12}\cdot n^{\frac12}\cdot g_{a,c,k,\nu}.\] 
\end{Lem}
\begin{proof}
Let us first treat the case when $ k $ is odd and we encounter no poles. We have $ \widetilde{k}=1 $ and 
\[I_{a,c,k,\nu}=\int_{\bb R}e^{-\frac{2\pi zx^2}{k}}H_{a,c}\left( \frac{2\pi i\nu}{k}-\frac{2\pi zx}{k}-\frac{\pi i}{2k} \right) dx.\] 
If we write $ \frac{\pi z}{k}=re^{i\phi} $ with $ r>0, $ then $ |\phi|<\frac{\pi }{2} $ since $ {\rm Re}(z)>0. $ The substitution $ \tau=\frac{\pi zx}{k} $ yields 
\begin{equation}\label{integralHac}
z^{\frac12}\cdot I_{a,c,k,\nu}(z)=\frac{k}{\pi z^{\frac12}}\int_L e^{-\frac{2k\tau^2}{\pi z}}H_{a,c}\left( \frac{2\pi i\nu}{k}-\frac{\pi i}{2
k} -2\tau\right) d\tau,
\end{equation}
where $ L $ is the line passing through 0 at an angle of argument $ \pm \phi. $
One easily sees that, 
for $0\le t\le \phi,  $
\begin{equation*}
\left| e^{-\frac{2kR^2e^{2it}}{\pi z}}H_{a,c}\left( \frac{2\pi i\nu}{k}-\frac{\pi i}{2
k} \pm2Re^{it}\right) dx\right| \to 0\quad{\text{as~}R\to\infty}.
\end{equation*}
As the integrand from \eqref{integralHac} has no poles, we can shift the path $ L $ of integration to the real line and obtain 
\begin{equation*}
z^{\frac12}\cdot I_{a,c,k,\nu}(z)=\frac{k}{\pi z^{\frac12}}\int_{\R} e^{-\frac{2kt^2}{\pi z}}H_{a,c}\left( \frac{2\pi i\nu}{k}-\frac{\pi i}{2
k} -2t\right) dt.
\end{equation*}
The inequality
\[\left| \sinh\left(\frac{\pi i\nu}{k}-\frac{\pi i}{4k}-t\pm\frac{\pi i a}{c} \right)  \right|\ge 
\left| \sin\left(\frac{\pi \nu}{k}-\frac{\pi }{4k}\pm\frac{\pi  a}{c} \right)  \right| \]
follows immediately for $ t\in\R  $ from the definition of $ \sinh $ and some easy manipulations, while the estimate
\[\left| \sin\left(\frac{\pi \nu}{k}-\frac{\pi }{4k}-\frac{\pi  a}{c} \right)  \right|\left| \sin\left(\frac{\pi \nu}{k}-\frac{\pi }{4k}+\frac{\pi  a}{c} \right)  \right|\gg_c \min\left\lbrace\left| \frac{\nu}{k}-\frac{1}{4k}+\frac ac\right|  , \left| \frac{\nu}{k}-\frac{1}{4k}-\frac ac \right| \right\rbrace\]
is clear. Therefore we have
\[z^{\frac12}\cdot I_{a,c,k,\nu}(z)\ll_c \frac k{\min\left\lbrace\left\lbrace\frac{\nu}{k}-\frac{1}{4k}+\frac ac \right\rbrace , \left\lbrace\frac{\nu}{k}-\frac{1}{4k}-\frac ac \right\rbrace\right\rbrace|z|^{\frac12}}\int_{\bb R}e^{-\frac{2k}{\pi }\RE \left(\frac1z \right) t^2}dt,\]
and, noting that 
\[\Big| e^{-\frac{2kt^2}{\pi z}}\Big|=e^{-\frac{2k}{\pi }\RE \left(\frac1z \right) t^2},\quad\RE\left( \frac 1z\right)^{-\frac12}\cdot |z|^{-\frac 12}\le \sqrt2 \cdot \sqrt n\cdot k^{-1}, \] the claim follows on making the substitution $ t\mapsto \sqrt{\frac{2k\RE \left( \frac1z\right) }{\pi}}\cdot t.$ 
\par If $ k $ is even and $ c\nmid k, $ then we proceed similarly as above. If, however, the integrand in \eqref{Mordell}  has a pole at $ x=0, $ in both of the cases $ c\mid k $ and $ c\nmid k, $ instead of $ \R, $ we must consider the path of integration to be the real line indented below 0. \par For simplicity, let us present the case when $ 2\mid k, $ as the case $ 2\nmid k $ is completely analogous. After doing the same change of variables as before and (if needed) shifting the path of integration (which will now consist of a straight line passing through 0 at an angle $ \pm\phi $ with a small segment centered at 0 removed and replaced by a semicircle inclined also at an angle $ \pm\phi $), the new path of integration will be given by $ \gamma_{R,\varepsilon}= [-R,-\varepsilon]\cup C_{\varepsilon}\cup [R,\varepsilon], $ where $ C_{\varepsilon} $ is the positively oriented semicircle of radius $ \varepsilon $ around 0 below the real line and 
\begin{equation*}\label{PP}
I_{a,c,k,\nu}=\frac k{\pi z}\int_{\gamma_{R,\varepsilon}}e^{-\frac{2 kt^2}{\pi z}}H_{a,c}\left( \frac{2\pi i\nu}{k}-2t \right) dt=\frac k{4\pi z}\int_{\gamma_{R,\varepsilon}}\frac{e^{-\frac{2 kt^2}{\pi z}}}{\sinh (t)\sinh\left(t-\frac{2\pi ia}{c} \right) }dt.
\end{equation*}
If we let $ D_{R,\varepsilon} $ be the enclosed path of integration $ \gamma_{R,\varepsilon}\cup [R,R+\pi ia/c]\cup [R+\pi ia/c,-R+\pi ia/c]\cup [-R+\pi ia/c,-R]$ and we set \[f(w):=\frac{e^{-\frac{2 kw^2}{\pi z}}}{\sinh (w)\sinh\left(w-\frac{2\pi ia}{c} \right) },\] then by the Residue Theorem we obtain 
\[\frac{4\pi z}{k}\cdot I_{a,c,k,\nu}=-\frac{2\pi}{\sin\left( \frac{2\pi a}{c}\right) }+\left( \int_{-R+\pi ia/c}^{-R}+\int_{-R+\pi ia/c}^{R+\pi ia/c}+\int_{R+\pi ia/c}^{R}\right) \frac{e^{-\frac{2 kw^2}{\pi z}}}{\sinh (w)\sinh\left(w-\frac{2\pi ia}{c} \right) }dw, \]
since inside and on $ D_{R,\varepsilon} $ the only pole of $ f $ is at $ w=0, $ with residue $$ \Res\limits_{w=0}f(w)=\frac i{\sin\left( \frac{2\pi a}{c}\right)}. $$
On $ [-R+\pi ia/c,-R] $ and $ [R+\pi ia/c,R] $ we have $ \left| \sinh(w) \sinh\left(w-\frac{2\pi ia}{c} \right)\right|\ge \sinh^2 R $ and $ \Big|e^{-\frac{2kw^2}{\pi z}} \Big|  =e^{-\frac{2k}{\pi }\RE\left(\frac1z \right)R^2 },$ thus the two corresponding integrals tend to 0 as $ R\to 0, $ whereas on $ [-R+\pi ia/c,R+\pi ia/c] $ we have, after a change of variables,
\[\int_{-R+\pi ia/c}^{R+\pi ia/c} \frac{e^{-\frac{2kw^2}{\pi z}}}{\sinh (w)\sinh\left(w-\frac{2\pi ia}{c} \right) }dw=\int_{-R}^{R} \frac{e^{-\frac{2k\left(t+\frac{\pi ia}{c} \right)^2}{\pi z}}}{\sinh \left(t+\frac{\pi ia}{c} \right) \sinh\left(t-\frac{\pi ia}{c} \right) }dt. \]
Proceeding now along the same lines as before, we obtain 

\[z^{\frac12}\cdot I_{a,c,k,\nu}(z)\ll \left( \frac {\pi a}c\right)^{-1}  \cdot \frac k{   |z|^{\frac12}}\int_{\bb R}e^{-\frac{2k}{\pi }\RE \left(\frac1z \right) t^2}dt,\] 
and the proof is complete. 
\end{proof}

\subsection{Kloosterman sums} The following is a variation of \cite[Lemma 4.1]{AndrewsThesis}, cf. Bringmann \cite[Lemma 3.2]{B}.
\begin{Lem}\label{estimateKloosterman}
Let $ m,n\in\bb Z , $ $ 0\le\sigma_1<\sigma_2\le k $ and $ D\in\bb Z $ with $ (D,k)=1. $
\begin{enumerate}[{\rm(i)}]
\item We have 
\begin{equation}\label{ineqKloosterman}
\primesum\sum_{\substack{h\\\sigma_1\le Dh' \le \sigma_2}}\frac{\omega^2_{h,k}}{\omega_{2h,k}}\cdot e^{\frac{2\pi i }{k}(hn+h'm)}\ll (24n+1,k)^{\frac12}\cdot k^{\frac{1}{2}+\varepsilon}.
\end{equation}
\item If $ c\mid k,$ we have 
\begin{equation}\label{ourineqKloosterman}
\tan\left(\frac{\pi a}{c} \right) \primesum\sum_{\substack{h\\\sigma_1\le Dh'\le\sigma_2}}\frac{\omega^2_{h,k}}{\omega_{2h,k}}\frac{1}{\sin\left(\frac{\pi ah'}{c} \right)}\cdot e^{-\frac{2\pi ih'a^2k_1}{c}}\cdot e^{\frac{2\pi i}{k}(nh+mh')}\ll (24n+1,k)^{\frac12}\cdot k^{\frac{1}{2}+\varepsilon}.
\end{equation} 
\item If $ c\mid k,$ we have
\begin{equation}\label{ourineqKloosterman2}
 \tan\left(\frac{\pi a}{c} \right) \primesum\sum_{\substack{h\\\sigma_1\le Dh'\le\sigma_2}}\frac{\omega^2_{h,k}}{\omega_{2h,k}}(-1)^{k_1+1}\cot\left(\frac{\pi ah'}{c} \right)\cdot e^{-\frac{2\pi ih'a^2k_1}{c}}\cdot e^{\frac{2\pi i}{k}(nh+mh')}\ll (24n+1,k)^{\frac12}\cdot k^{\frac{1}{2}+\varepsilon}.
\end{equation} 
\end{enumerate} 
The implied constants are independent of $ a $ and $k, $ and $ \varepsilon>0 $ can be taken arbitrarily.

\end{Lem}
\begin{proof}
Part (i) follows simply on replacing $ \omega_{h,k} $ by $\frac{\omega^2_{h,k}}{\omega_{2h,k}}  $ in the proof of Andrews \cite[Lemma 4.1]{AndrewsThesis}. As the proof of \eqref{ourineqKloosterman2} is completely analogous to that of \eqref{ourineqKloosterman}, we deal only with part (ii).   
We set $ \widetilde{c}=c $ if $ k  $ is odd, and $ \widetilde{c}=2c $ if $ k $ is even. 
Since $ e^{-\frac{2\pi ih'a^2k_1}{c}} $ depends only on the residue class of $ h'$ modulo $\widetilde{c}, $ the left-hand side of \eqref{ourineqKloosterman} can be rewritten as \[\tan\left(\frac{\pi a}{c} \right)\sum_{c_j}\frac{e^{-\frac{2\pi ia^2k_1c_j}{c}}}{\sin\left(\frac{\pi ac_j}{c} \right)}  \sideset{}{'}\sum_{\substack{h\\\sigma_1\le Dh'\le\sigma_2\\h'\equiv c_j\pmod*{\widetilde{c}}}}\frac{\omega^2_{h,k}}{\omega_{2h,k}}\cdot e^{\frac{2\pi i}{k}(nh+mh')}, \]
where $ c_j $ runs over a set of primitive residues modulo $ \widetilde{c}. $ Furthermore, we have  
\begin{align*} \primesum\sum_{\substack{h\\\sigma_1\le Dh'\le\sigma_2\\h'\equiv c_j\pmod*{\widetilde{c}}}}\frac{\omega^2_{h,k}}{\omega_{2h,k}}\cdot e^{\frac{2\pi i}{k}(nh+mh')}&=\frac 1{\widetilde{c}}\sideset{}{'}\sum_{\substack{h\\\sigma_1\le Dh'\le\sigma_2}}\frac{\omega^2_{h,k}}{\omega_{2h,k}}\cdot e^{\frac{2\pi i}{k}(nh+mh')}\sum_{r\pmod*{\widetilde c}} e^{\frac{2\pi ir}{\widetilde c}(h'-c_j)}\\ &=\frac 1{\widetilde{c}}\sum_{r\pmod*{\widetilde c}}e^{-\frac{2\pi irc_j}{\widetilde c}(h'-c_j)}\primesum\sum_{\substack{h\\\sigma_1\le Dh'\le\sigma_2\\h'\equiv c_j\pmod*{\widetilde{c}}}}\frac{\omega^2_{h,k}}{\omega_{2h,k}}\cdot e^{\frac{2\pi i}{k}\left( nh+\left( m+\frac{kr}{\widetilde c}\right) h'\right) }
\end{align*}
and the proof is concluded on invoking part (i) and noting that $ \frac{kr}{\widetilde c}\in\bb Z. $   
\end{proof}

\section{Asymptotics for $ A\left(\frac ac;n \right)  $ and $ N(a,c,n) $}\label{Proof of Theorem 1}
We turn our focus now to the proof of Theorem \ref{MainThm} and proceed as described in the strategy outlined in Section \ref{Strategy}, the whole section being dedicated to this purpose.
\begin{proof}[Proof of Theorem \ref{MainThm}]
On using Cauchy's Theorem and splitting the path of integration into Farey arcs as explained in Section \ref{SubsectionCircleMethod}, we obtain, from \eqref{CircleFarey} and Theorem \ref{Transformations}, 
\begin{multline*}
A\left(\frac ac;n \right)=i\tan \left(\frac{\pi a}{c} \right) \sum_{\substack{h,k\\2|k,~c|k}}\frac{\omega^2_{h,k}}{\omega_{h,k/2}}(-1)^{k_1+1}\cot\left(\frac{\pi ah'}{c} \right)e^{-\frac{2\pi ia^2h'k_1}{c}-\frac{2\pi inh}{k}}\int_{-\vartheta'_{h,k}}^{\vartheta''_{h,k}} z^{-\frac12} e^{\frac{2\pi nz}{k}} \cal O\left(\frac{ah'}{c}; q_1\right)d\Phi\\
-\sqrt2i\tan \left(\frac{\pi a}{c} \right) \sum_{\substack{h,k\\2\nmid k,~c|k}}\frac{\omega^2_{h,k}}{\omega_{2h,k}}e^{\frac{\pi ih'}{8k}-\frac{2\pi ia^2h'k_1}{c}-\frac{2\pi inh}{k}}\int_{-\vartheta'_{h,k}}^{\vartheta''_{h,k}} z^{-\frac12} e^{\frac{2\pi nz}{k}} \cal U\left(\frac{ah'}{c}; q_1\right)d\Phi\\
+2\tan \left(\frac{\pi a}{c} \right) \sum_{\substack{h,k\\2|k,~c\nmid k,~c_1\ne 2}}\frac{\omega^2_{h,k}}{\omega_{h,k/2}}(-1)^{c_1(\ell+k_1)}e^{-\frac{2\pi ia^2h'k_1}{c_1c}-\frac{2\pi inh}{k}}\int_{-\vartheta'_{h,k}}^{\vartheta''_{h,k}} z^{-\frac12} e^{\frac{2\pi nz}{k}} \cal O\left(ah',\frac{\ell c}{c_1},c; q_1\right)d\Phi\\
+\tan \left(\frac{\pi a}{c} \right) \sum_{\substack{h,k\\2| k,~c\nmid k,~c_1= 2}}\frac{\omega^2_{h,k}}{\omega_{h,k/2}}e^{-\frac{\pi ia^2h'k_1}{c}-\frac{2\pi inh}{k}}\int_{-\vartheta'_{h,k}}^{\vartheta''_{h,k}} z^{-\frac12} e^{\frac{2\pi nz}{k}} \cal V\left(\frac{ah'}{c}; q_1\right)d\Phi\\
+\sqrt2\tan \left(\frac{\pi a}{c} \right) \sum_{\substack{h,k\\2\nmid k,~c\nmid k,~c_1\ne 4}}\frac{\omega^2_{h,k}}{\omega_{2h,k}}e^{\frac{\pi ih'}{8k}-\frac{2\pi ia^2h'k_1}{c_1c}-\frac{2\pi inh}{k}}\int_{-\vartheta'_{h,k}}^{\vartheta''_{h,k}} z^{-\frac12} e^{\frac{2\pi nz}{k}} \cal U\left(ah',\frac{\ell c}{c_1},c; q_1\right)d\Phi\\
+\frac1{\sqrt2}\tan \left(\frac{\pi a}{c} \right) \sum_{\substack{h,k\\2\nmid k,~c\nmid k,~c_1= 4}}\frac{\omega^2_{h,k}}{\omega_{h,k/2}}e^{\frac{\pi ih'}{8k}-\frac{2\pi ia^2h'k_1}{c_1c}-\frac{2\pi inh}{k}}\int_{-\vartheta'_{h,k}}^{\vartheta''_{h,k}} z^{-\frac12} e^{\frac{2\pi nz}{k}} \cal V\left(ah',\frac{\ell c}{c_1},c; q_1\right)d\Phi\\
+4\sin^2\left(\frac{\pi a}{c}\right) \sum_{\substack{h,k\\2|k}}\frac{\omega^2_{h,k}}{\omega_{h,k/2}\cdot k} e^{-\frac{2\pi inh}{k}} \sum_{\nu=0}^{k-1}(-1)^{\nu}e^{-\frac{2\pi ih'\nu^2}{k}}\int_{-\vartheta'_{h,k}}^{\vartheta''_{h,k}} z^{\frac12} e^{\frac{2\pi nz}{k}} I_{a,c,k,\nu}(z)d\Phi\\
+4\sqrt2\sin^2\left(\frac{\pi a}{c}\right) \sum_{\substack{h,k\\2\nmid k}}\frac{\omega^2_{h,k}}{\omega_{2h,k}\cdot k} e^{-\frac{2\pi inh}{k}} \sum_{\nu=0}^{k-1}e^{-\frac{\pi ih'}{k}(2\nu^2-\nu)}\int_{-\vartheta'_{h,k}}^{\vartheta''_{h,k}} z^{\frac12} e^{\frac{2\pi nz}{k}} I_{a,c,k,\nu}(z)d\Phi\\
=:\sum_1+\sum_2+\sum_3+\sum_4+\sum_5+\sum_6+\sum_7+\sum_8.
\end{multline*}
For the reader's convenience, we divide our proof into several steps.
We start by estimating the sums $ \sum_2,\sum_5 $ and $ \sum_6, $ which, as we shall see, will give the main contribution. The sums $ \sum_1,\sum_3,\sum_4,\sum_7 $ and $ \sum_8 $ will go into an error term and will be dealt with at the end. Here the analysis will also split, as the latter two sums can be treated together. 
\subsection{Estimates for the sums $\sum_2,\sum_5 $ and $ \sum_6 $}
	
To estimate $ \sum_2,  $ notice that we can write 
\begin{align*}
\sum_{n\in\bb Z}\frac{(1+q^n)q^{n^2+\frac  n2}}{1-2\cos\left(\frac{2\pi a}{c} \right)q^n+q^{2n} } &=\frac{1}{2\sin^2\left(\frac{\pi a}{c} \right) }+2\sum_{n\ge 1}\frac{(1+q^n)q^{n^2+\frac  n2}}{1-2\cos\left(\frac{2\pi a}{c} \right)q^n+q^{2n} }\\
&=\frac{1}{2\sin^2\left(\frac{\pi a}{c} \right) }+2\sum_{2|n }\frac{(1+q^n)q^{n^2+\frac  n2}}{1-2\cos\left(\frac{2\pi a}{c} \right)q^n+q^{2n} }+2q^{\frac12}\sum_{2\nmid n }\frac{(1+q^n)q^{n^2+\frac  {n-1}2}}{1-2\cos\left(\frac{2\pi a}{c} \right)q^n+q^{2n} }\\&=\frac{1}{2\sin^2\left(\frac{\pi a}{c} \right) }+\sum_{r\ge1}a_2(r)e^{\frac{2\pi i m_rh'}{k}-\frac{2\pi r}{kz}}+q^{\frac12}\sum_{r\ge1}b_2(r)e^{\frac{2\pi i n_rh'}{k}-\frac{2\pi r}{kz}},
\end{align*}
where $ m_r,n_r\in \bb Z $ and the coefficients $ a_2(r)$ and $b_2(r) $ are independent of $ k $ and $ h. $
\par On replacing $ z $ by $ z_1=z/2 $ in \eqref{GenFunOver}, we have 
\begin{align*}
\cal U\left(\frac{ah'}{c}; q_1\right)&=\sin\left(\frac{\pi ah'}{c} \right)\frac{\eta\left(\frac{z_1}2 \right) }{\eta^2(z_1)}\sum_{n\in\Z}\frac{(1+q_1^n)q_1^{n^2+\frac n2}}{1-2q_1^n\cos\left(\frac{2\pi ah'}{c} \right)+q_1^{2n} }\\
&=\sin\left(\frac{\pi ah'}{c} \right)\frac{\eta\left(\frac{z_1}2 \right) }{\eta^2(z_1)}\Bigg( \frac1{2\sin^2\left( \frac{\pi ah'}{c}\right) }+2\sum_{n\ge1}\frac{(1+q_1^n)q_1^{n^2+\frac n2}}{1-2q_1^n\cos\left(\frac{2\pi ah'}{c} \right)+q_1^{2n} }\Bigg)\\
&=q_1^{-\frac1{16}}\cdot \ol P(q_1)\cdot\Bigg(\frac{1}{2\sin\left(\frac{\pi ah'}{c} \right) } +2\sin\left(\frac{\pi ah'}{c} \right)\sum_{n\ge1}\frac{(1+q_1^n)q_1^{n^2+\frac n2}}{1-2q_1^n\cos\left(\frac{2\pi ah'}{c} \right)+q_1^{2n} }\Bigg), \end{align*}
where we write $ q_1=e^{2\pi iz_1}. $ It follows that
\begin{align*}
\sum_2&=-\sqrt2i\tan \left(\frac{\pi a}{c} \right) \sum_{\substack{h,k\\2\nmid k,~c|k}}\frac{\omega^2_{h,k}}{\omega_{2h,k}}e^{\frac{\pi ih'}{8k}-\frac{2\pi ia^2h'k_1}{c}-\frac{2\pi inh}{k}}\int_{-\vartheta'_{h,k}}^{\vartheta''_{h,k}} z^{-\frac12}\cdot e^{\frac{2\pi nz}{k}}\cdot \cal U\left(\frac{ah'}{c}; q_1\right)d\Phi\\
&=-\frac{i}{\sqrt2}\tan\left(\frac{\pi a}{c} \right) \sum_{2\nmid k,~c|k}\primesum\sum_{h}\frac{\omega^2_{h,k}}{\omega_{2h,k}}\frac{1}{\sin\left(\frac{\pi ah'}{c} \right)}e^{-\frac{2\pi ia^2h'k_1}{c}-\frac{2\pi inh}{k}}\int_{-\vartheta'_{h,k}}^{\vartheta''_{h,k}} z^{-\frac12}\cdot e^{\frac{2\pi nz}{k}+\frac{\pi}{8kz}}\cdot \widetilde{\cal U}\left(\frac{ah'}{c}; q_1\right)d\Phi,
\end{align*}
with
\begin{equation*}
\widetilde{\cal U}\left(\frac{ah'}{c}; q_1\right)  =1+4\sin^2\left(\frac{\pi ah'}{c} \right)\Bigg(\sum_{r\ge1}{a}_2(r)e^{\frac{2\pi i {m}_rh'}{k}-\frac{2\pi r}{kz}}+q^{\frac12}\sum_{r\ge1}{b}_2(r)e^{\frac{2\pi i {n}_rh'}{k}-\frac{2\pi r}{kz}}\Bigg).
\end{equation*}
\par We treat the sum coming from the constant term and the two sums coming from the case $ r\ge1 $ separately. The former will contribute to the main term, while the latter two sums will contribute to the error term. We denote the associated sums by $ S_1 $, $ S_2 $ and $ S_3 $ and we first estimate $ S_2 $ ($ S_3 $ is dealt with in a similar manner).\par We recall, from Section \ref{SubsectionCircleMethod}, the easy facts that  
\begin{equation}\label{boundsFarey}
{\rm Re}(z)=\frac kn,\quad  {\rm Re}\left( \frac 1z\right)>\frac k2,\quad |z|^{-\frac12}\le n^{\frac12}\cdot k^{-\frac12} \quad\text{and}\quad \vartheta_{h,k}'+\vartheta_{h,k}''\le \frac{2}{k(N+1)}.
\end{equation} 
\par We write
\begin{equation}\label{splitintegralRademacher}
\int_{-\vartheta_{h,k}'}^{\vartheta_{h,k}''}=\int_{-\frac1{k(N+k)}}^{\frac1{k(N+k)}}+\int_{-\frac1{k(k_1+k)}}^{-\frac1{k(N+k)}}+\int_{\frac1{k(N+k)}}^{\frac1{k(k_2+k)}}\end{equation}
and denote the associated sums by $ S_{21}, $ $ S_{22} $ and $ S_{23}. $
This way of splitting the integral is motivated by the Farey dissection used by Rademacher \cite[pp. 504--509]{Rademacher}. It allows us to interchange summation with the integral and yields a so-called \textit{complete} Kloosterman sum and two \textit{incomplete} Kloosterman sums. Lemma \ref{estimateKloosterman} applies to both types of sums.
\par We first consider $ S_{21}. $ As we have already seen,
\begin{equation*}
\overline{p}(n)\sim\frac1{8n}e^{\pi\sqrt n},
\end{equation*}  
thus 
$\overline{p}(n)<e^{\pi\sqrt n}$ as $ n\to\infty. $ Clearly, the coefficients of $ \cal O(u;q), $ regarded as a series in $ q $ when evaluated at a root of unity $ u=\zeta_c^a, $ satisfy \[\left| A\left(\frac ac;n \right)\right| \le \sum_{m=-\infty}^{\infty}|\ol N(m,n)\zeta_c^{am}|\le \sum_{m=-\infty}^{\infty}\ol N(m,n)=\ol p(n), \]
thus, in light of the transformation behavior shown in Theorem \ref{Transformations}, the coefficients $ a_2(r) $ and $ b_2(r) $ satisfy 
\begin{equation}
\label{asympoverpartitions}
|a_2(r)|,~|b_2(r)|\le e^{\pi\sqrt r}\quad\text{as~}r\to\infty.
\end{equation}
\par As the integral that appears inside the sum does not depend on $ h, $ in evaluating $ S_{21} $ we can perform summation with respect to $ h. $ Using, in turn, the bound \eqref{asympoverpartitions}, Lemma \ref{estimateKloosterman},  the estimates from \eqref{boundsFarey}, and the well-known bound $ \sigma_0(n)=o(n^{\epsilon}) $ for all $ \epsilon>0, $ we obtain
\begin{align*}
S_{21}&\ll \Bigg| \sum_{r=1}^{\infty}a_2(r) \sum_{c|k} \tan\left(\frac{\pi a}{c} \right)\primesum \sum_h \frac{\omega^2_{h,k}}{\omega_{2h,k}}\cdot\frac{1}{\sin\left(\frac{\pi ah'}{c} \right)}\cdot e^{-\frac{2\pi ih'a^2k_1}{c}-\frac{2\pi inh}{k}+\frac{2\pi i m_rh'}{k}}  \\
&\phantom{\ll \left| \sum_{r=1}^{\infty}a_2(r) \sum_{c|k} \tan\left(\frac{\pi a}{c} \right)\primesum \sum_h \frac{\omega^2_{h,k}}{\omega_{2h,k}}\cdot\frac{1}{\sin\left(\frac{\pi ah'}{c} \right)}\right.~} \cdot \int_{-\frac1{k(N+k)}}^{\frac1{k(N+k)}}z^{-\frac12}\cdot e^{-\frac{2\pi }{kz}\left( r-\frac1{16}\right) +\frac{2\pi zn}{k}}d\Phi  \Bigg|\\
&\ll \sum_{r=1}^{\infty} | a_2(r)|\cdot e^{-\pi r}\sum_k k^{-1+\varepsilon}\cdot(24n+1,k)^{\frac12}\ll \sum_{\substack{d|24n+1\\d\le N}}d^{\frac12}\sum_{k\le\frac Nd}(dk)^{-1+\varepsilon}\\
&\ll \sum_{\substack{d|24n+1\\d\le N}}d^{-\frac12+\varepsilon}\int_{1}^{N/d}x^{-1+\varepsilon}dx= \sum_{\substack{d|24n+1\\d\le N}}d^{-\frac12}\cdot d^{\varepsilon}\cdot\left(\frac{N}{d} \right)^{\varepsilon} \ll \sum_{d|24n+1}d^{-\frac12}\cdot n^{\frac{\varepsilon}2} \\
&\ll n^{\epsilon+\frac{\varepsilon}2} \ll n^{\varepsilon}.
\end{align*}
For a proof of the fact that $ \sigma_0(n)=o(n^{\epsilon}) $ see, e.g., \cite[p. 296]{Apo}. Here we bound trivially $$\sum_{d|24n+1}d^{-\frac12}<\sum_{d|24n+1}1=\sigma_0(24n+1)=o(n^{\epsilon}) $$ and choose $ 0<\epsilon<\varepsilon/2, $ where $ \sigma_0(n) $ denotes, as usual, the number of divisors of $ n. $  
\par Since $ S_{22} $ and $ S_{23} $ are treated in the exact same way, we only consider $ S_{22}. $ Writing 
\[\int_{-\frac1{k(k_1+k)}}^{-\frac1{k(N+k)}}=\sum_{\ell=k_1+k}^{N+k-1}\int_{-\frac1{k\ell}}^{-\frac1{k(\ell+1)}}\] we see that 
\begin{multline}
\label{estimateS22}
S_{22}\ll \Bigg| \sum_{r=1}^{\infty} a_2(r) \sum_{c|k} \sum_{\ell=k_1+k}^{N+k-1}\int_{-\frac1{k\ell}}^{-\frac1{k(\ell+1)}} z^{-\frac12}\cdot e^{-\frac{2\pi }{kz}\left( r-\frac1{16}\right) +\frac{2\pi zn}{k}}d\Phi  \nonumber\\ \cdot   \tan\left(\frac{\pi a}{c} \right) \primesum\sum_{\substack{h\\N<k+k_1\le \ell}} \frac{\omega^2_{h,k}}{\omega_{2h,k}}\cdot\frac{1}{\sin\left(\frac{\pi ah'}{c} \right)}\cdot e^{-\frac{2\pi ih'a^2k_1}{c}-\frac{2\pi inh}{k}+\frac{2\pi i m_rh'}{k}} \Bigg|.
\end{multline}
Again, from basic facts of Farey theory, it follows that \[N-k<k_1,k_2\le N\quad\text{and}\quad k_1\equiv-k_2\equiv -h'\pmod*k,\]
conditions which imply the restriction of $ h' $ to one or two intervals in the range $ 0\le h'<k. $
Therefore we can use Lemma \ref{estimateKloosterman} to estimate the above expression just as in the case of $ S_{21}. $ \par As for the estimation of $ S_1, $ we can split the integration path into 
\[\int_{-\vartheta_{h,k}'}^{\vartheta_{h,k}''}=\int_{-\frac1{kN}}^{\frac1{kN}}-\int_{-\frac1{kN}}^{-\frac1{k(k_1+k)}}-\int_{\frac1{k(k_2+k)} }^{\frac1{kN}}\]
and denote the associated sums by $ S_{11}, $ $ S_{12} $ and $ S_{13}. $ The sums $ S_{12} $ and $ S_{13} $ contribute to the error term and, since they are of the same shape, we only consider $ S_{12}. $ Further, decomposing
\[\int_{-\frac1{kN}}^{-\frac1{k(k_1+k)}}=\sum_{\ell=N}^{k_1+k-1}\int_{-\frac1{k\ell}}^{-\frac1{k(\ell+1)}}\] gives 
\begin{multline}
S_{12}\ll \left|  \sum_{c|k}\sum_{\ell=N}^{k_1+k-1}\int_{-\frac1{k\ell}}^{-\frac1{k(\ell+1)}}  z^{-\frac12}\cdot e^{\frac{\pi }{8kz}+\frac{2\pi zn}{k}}d\Phi\right.  \nonumber\\ \cdot\left.   \tan\left(\frac{\pi a}{c} \right) \primesum\sum_{\substack{h\\\ell<k_1+k-1\le N+k-1}} \frac{\omega^2_{h,k}}{\omega_{2h,k}}\cdot\frac{1}{\sin\left(\frac{\pi ah'}{c} \right)}\cdot e^{-\frac{2\pi ih'a^2k_1}{c}-\frac{2\pi inh}{k}} \right|.
\end{multline}
Using the facts that $$ {\rm Re}(z)=\frac kn, \quad {\rm Re}\left(\frac 1z \right)<k \quad \text{and}\quad  |z|^2\ge \frac{k^2}{n^2}, $$ this sum can be estimated as before against $ O(n^{\varepsilon}). $ Thus, 
\begin{equation}\label{sigma2}
\sum_2=i\sum_{c|k}B_{a,c,k}(-n,0)\int_{-\frac1{kN}}^{\frac1{kN}}z^{-\frac12}\cdot e^{\frac{2\pi zn}{k}+\frac{\pi}{8kz}}d\Phi+O(n^{\varepsilon}).
\end{equation} 
\par We stop here for the moment with the estimation of $ \sum_2 $ and turn our attention to $ \sum_5. $ This sum is treated in a similar manner, but some comments regarding necessary modifications are in order. On noting that 
\begin{equation}\label{expressionU}
\cal U\left(a,b,c; q\right)=\frac{\eta\left(\frac z2 \right) }{\eta^2(z)}e^{\frac{\pi i a}{c}\left(\frac{4b}{c}-2s \right) }q^{\frac{s b}{c}-\frac{b^2}{c^2}}\Bigg(\sum_{m\ge0}\frac{e^{-\frac{\pi ia}{c}}q^{\frac{m(2m+1)}{2}+ms+\frac{b}{2c}} }{1-e^{-\frac{2\pi i a}{c}}q^{m+\frac bc}} -\sum_{m\ge1}\frac{e^{\frac{\pi ia}{c}}q^{\frac{m(2m+1)}{2}-ms-\frac{b}{2c}} }{1-e^{\frac{2\pi i a}{c}}q^{m-\frac bc}} \Bigg) ,
\end{equation}
we see that 
\begin{equation}\label{sum5}
e^{\frac{\pi ih'}{8k}-\frac{2\pi ia^2h'k_1}{c_1c}}\cdot \cal U\left(ah',\frac{\ell c}{c_1},c; q_1\right)
=\sum_{r\ge r_0}a_5(r)e^{\frac{2\pi im_rh'}{k}}e^{-\frac{\pi r}{kc_1^2z}}+e^{\frac{\pi i h'}{k}}\sum_{r\ge r_1}b_5(r)e^{\frac{2\pi in_rh'}{k}}e^{-\frac{\pi r}{kc_1^2z}},
\end{equation}
where $ m_r,n_r, r_0,r_1\in\bb Z. $ By the same argument as for $ S_{21}, $ one sees immediately that the part which might contribute to the main term can come only from those terms with $ r<0. $ A straightforward, but rather tedious, computation shows that such terms can arise only for $ s=0, $ $ m=0 $ in the first sum, respectively for  $ s=2, $ $ m=1 $ in the second sum obtained by expressing $\cal U\left(ah',\frac{\ell c}{c_1},c; q_1\right) $ as shown in \eqref{expressionU}.
In the former case, the contribution is given by 
\[e^{-\frac{2\pi ia^2h'k_1}{c_1c}+\frac{4\pi i ah'\ell}{c_1c}-\frac{\pi iah'}{c}}\cdot q_1^{-\frac1{16}-\frac{\ell^2}{c_1^2}+\frac{\ell}{2c_1}}\sum_{\substack{r\\\delta_{c,k,r}>0}}e^{-\frac{2\pi iah'r}{c}}\cdot q_1^{\frac{\ell r}{c_1}},\]
and, in the latter, by 
\[-e^{-\frac{2\pi ia^2h'k_1}{c_1c}+\frac{4\pi i ah'\l}{c_1c}-\frac{3\pi iah'}{c}}\cdot q_1^{-\frac1{16}-\frac{\ell^2}{c_1^2}+\frac{3\ell}{2c_1}-\frac12}\sum_{\substack{r\\\delta_{c,k,r}>0}}e^{\frac{2\pi iah'r}{c}}\cdot q_1^{\left( 1-\frac{\ell }{c_1}\right) r}.\]
To evaluate $ \sum_5, $ note that one can split the sum over $ k $ into groups based on the value $ k_1, $ which is defined in terms of $ c_1 $ and $ \ell. $ In each such group, the value of $ \delta_{c,k,r} $ (hence the condition $ \delta_{c,k,r}>0 $) is independent of $ k, $ and the number of terms satisfying $ \delta_{a,c,k,r}>0 $ is finite and bounded in terms of $ c_1 $ (hence of $ c $). Moreover, the coefficients $ a_5(r) $ and $ b_5(r) $ are independent of $ k $ in any such fixed group. Since the terms with $ r<0 $ from \eqref{sum5} can be estimated as in the case of $ S_{2}, $  we obtain
\begin{equation}\label{sigma5}
\sum_5=\sqrt2\tan\left(\frac{\pi a}{c} \right)\sum_{\substack{k,r\\c\nmid k,~2\nmid k,~c_1\ne4\\\delta_{c,k,r}>0}}\primesum\sum_{h}\frac{\omega^2_{h,k}}{\omega_{2h,k}} e^{\frac{2\pi i}{k}(-nh+m_{a,c,k,r}h')}\cdot\int_{-\vartheta'_{h,k}}^{\vartheta''_{h,k}} z^{-\frac12}\cdot e^{\frac{2\pi nz}{k}+\frac{2\pi}{kz}\delta_{c,k,r}}d\Phi+O(n^{\varepsilon}),
\end{equation} 
with $ \delta_{c,k,r} $ and $ m_{a,c,k,r} $ as defined in \eqref{delta} and \eqref{mackr}. In a completely similar way, we compute 
\begin{align*}\label{sigma6}
\sum_6&=\frac1{\sqrt2}\tan\left(\frac{\pi a}{c} \right)\sum_{\substack{k,r\\c\nmid k,~2\nmid k,~c_1=4\\\delta'_{c,k,r}>0}}\primesum\sum_{h}\frac{\omega^2_{h,k}}{\omega_{2h,k}}\cdot e^{\frac{2\pi i}{k}(-nh+m_{a,c,k,r}h')}\cdot\int_{-\vartheta'_{h,k}}^{\vartheta''_{h,k}} z^{-\frac12}\cdot e^{\frac{2\pi nz}{k}+\frac{2\pi}{kz}\delta_{c,k,r}}d\Phi\nonumber\nonumber\\&\phantom{=~}+\frac1{\sqrt2}\tan\left(\frac{\pi a}{c} \right)\sum_{\substack{k,r\\c\nmid k,~2\nmid k,~c_1=4\\\delta''_{c,k,r}>0}}\primesum\sum_{h}\frac{\omega^2_{h,k}}{\omega_{2h,k}}\cdot e^{\frac{2\pi i}{k}(-nh+m'_{a,c,k,r}h')}\cdot\int_{-\vartheta'_{h,k}}^{\vartheta''_{h,k}} z^{-\frac12}\cdot e^{\frac{2\pi nz}{k}+\frac{2\pi}{kz}\delta'_{c,k,r}}d\Phi\\&\phantom{=~}+O(n^{\varepsilon}),\nonumber
\end{align*}
where we define 
\begin{equation*}\label{deltaprime}
\delta'_{c,k,r}:=\begin{cases}
 \frac{1}{16}-\frac{3 \ell}{2c_1}+\frac{\ell^2}{c_1^2}-r\frac{\ell}{c_1} &\text{{if~}}0<\frac{\ell}{c_1}\le\frac14,\\
\frac{1}{16}-\frac{3\ell}{2c_1}+\frac{\ell^2}{c_1^2}+\frac12-r\left(1-\frac{\ell}{c_1} \right) & \text{if~}\frac14<\frac{\ell}{c_1}\le\frac34,\\
\frac{1}{16}-\frac{5\ell}{2c_1}+\frac{\ell^2}{c_1^2}+\frac32-r\left(1-\frac{\ell}{c_1} \right)  & \text{if~}\frac34<\frac{\ell}{c_1}<1,
\end{cases}
\end{equation*}
%
and

\begin{equation*}\label{mackrprime}
m'_{a,c,k,r}:=\begin{cases}
-\frac{1}{2c_1^2}( 2(ak_1-\l)^2+3c_1(ak_1-\l)+2rc_1(ak_1-\ell) ) &\text{{if~}}0<\frac{\ell}{c_1}\le\frac14,\\
-\frac{1}{2c_1^2}( 2(ak_1-\l)^2+3c_1(ak_1-\l)-2rc_1(ak_1-\l)-c_1^2(2r-1) ) & \text{if~}\frac14<\frac{\ell}{c_1}\le\frac34,\\
-\frac{1}{2c_1^2}( 2(ak_1-\l)^2+5c_1(ak_1-\l)-2rc_1(ak_1-\l)-c_1^2(2r-3) ) & \text{if~}\frac34<\frac{\ell}{c_1}<1.
\end{cases}
\end{equation*}

An easy computation shows that if $ c_1=4, $ then $ \delta_{a,c,k,r}\le0 $ for all $ r\ge0, $ and that $ \delta_{a,c,k,r}'>0 $ if and only if $ r=0, $ $ m=1, $ $ s=1 $ and $ \l=2, $ case which is impossible as it leads to $ ak_1\equiv2\pmod*4, $ and by assumption $ k $ is odd, while the condition $ (a,c)=1 $ implies that $ a $ is odd as well. Therefore $ \sum_6 $ will only contribute to the error term.  
\par To finish the estimation of these sums, we are only left with computing integrals of the form 
\[\mathcal{I}_{k,v}=\int_{-\frac1{kN}}^{\frac1{kN}}z^{-\frac12}\cdot e^{\frac{2\pi }{k}\left(nz+\frac vz \right) }d\Phi,\]
which, upon substituting $ z=\frac kn-ik\Phi, $ are equal to 
\begin{equation}
\label{integralsIhr}
\mathcal{I}_{k,v}=\frac1{ki}\int_{\frac kn-\frac iN}^{\frac kn+\frac iN}z^{-\frac12}\cdot e^{\frac{2\pi }{k}\left(nz+\frac vz \right) } dz.
\end{equation}
\par To compute these integrals, we proceed in the way described by Dragonette \cite[p. 492]{Dragon} and made more precise by Bringmann \cite[p. 3497]{B}. In doing so, we enclose the path of integration by including the smaller arc of the circle through $ \frac kn\pm \frac iN $ and tangent to the imaginary axis at 0, which we denote by $ \Gamma. $ If $ z=x+iy ,$ then $ \Gamma $ is given by $ x^2+y^2=w x, $ with $ w=\frac kn+\frac n{N^2k}. $  Using the fact that $ 2>w>\frac 1k, $ $ \RE (z)\le \frac kn $ and $ \RE \left(\frac 1z \right)<k  $ on the smaller arc, the integral along this arc is seen to be of order $ O\big(n^{-\frac34}\big ).  $ By Cauchy's Theorem, the path of integration in \eqref{integralsIhr} can be further changed  into the larger arc of $ \Gamma, $ hence
\[\mathcal{I}_{k,v}=\frac1{ki}\int_{\frac kn-\frac iN}^{\frac kn+\frac iN}z^{-\frac12}\cdot e^{\frac{2\pi }{k}\left(nz+\frac vz \right) } dz+O\big( n^{-\frac18}\big). \]
Making the substitution $ t=\frac{2\pi v}{kz} $ gives 
\[\mathcal{I}_{k,v}=\frac{2\pi}{k}\left( \frac{2\pi v}{k}\right)^{\frac12}\frac1{2\pi i} \int_{\gamma-i\infty}^{\gamma+i\infty}t^{-\frac32}\cdot e^{t+\frac{\alpha }{t} } dt+O\big( n^{-\frac18}\big),\]
where $ \gamma\in\bb R $ and $ \alpha=\frac{4\pi^2 vn}{k^2}. $ Using the Hankel integral formula, we compute (see, e.g, \cite[\S 3.7 and \S 6.2]{Watson}) 
\begin{equation*}
\frac{1}{2\pi i}\int_{\gamma-i\infty}^{\gamma+i\infty}t^{-\frac32}\cdot e^{t+\frac{\alpha}{t}}dt=\frac1{\sqrt{\pi \alpha}}\cdot \sinh(2\sqrt{\alpha}),
\end{equation*}
hence 
\begin{equation}
\label{Ihk}
\mathcal{I}_{k,v}=\sqrt{\frac{2}{kn}}\cdot\sinh\left( \frac{4\pi \sqrt{vn} }{k}\right)+O\big( n^{-\frac18}\big).
\end{equation}
On applying \eqref{Ihk} to \eqref{sigma2} and \eqref{sigma5} for $ v=\frac1{16} $ and $ v=\delta_{c,k,r} $ respectively, we have
%
\begin{align*}\label{maincontribution}
\sum_2+\sum_5+\sum_6&=i\sqrt{\frac 2n}\sum_{\substack{1\le k\le\sqrt n\\2\nmid k,~c|k}}\frac{B_{a,c,k}(-n,0)}{\sqrt k}\cdot\sinh\left( \frac{\pi\sqrt n}{k}\right)\nonumber \\
&\phantom{=~}+2\sqrt{\frac 2n}\sum_{\substack{1\le k\le \sqrt n\\c\nmid k,~2\nmid k,~c_1\ne4\\r\ge0,~\delta_{c,k,r}>0}}\frac{D_{a,c,k}(-n,m_{a,c,k,r})}{\sqrt k}\cdot \sinh \left( \frac{4\pi\sqrt{\delta_{c,k,r}n}}{k} \right)+O(n^{\varepsilon}).
\end{align*}

\subsection{Estimates for the sums $ \sum_1,\sum_3$ and $\sum_4 $}
We show that these sums contribute only to the error term. Let us start our discussion with $ \sum_1, $ which equals 
\[\sum_1=i\tan \left(\frac{\pi a}{c} \right) \sum_{\substack{h,k\\2|k,~c|k}}\frac{\omega^2_{h,k}}{\omega_{h,k/2}}(-1)^{k_1+1}\cot\left(\frac{\pi ah'}{c} \right)e^{-\frac{2\pi ia^2h'k_1}{c}-\frac{2\pi inh}{k}}\int_{-\vartheta'_{h,k}}^{\vartheta''_{h,k}} z^{-\frac12}\cdot e^{\frac{2\pi nz}{k}}\cdot \cal O\left(\frac{ah'}{c}; q_1\right)d\Phi.\]
Although not written down explicitly in \cite{BJoy}, one can readily see, e.g., by inspecting the proof of Theorem 2.1 from \cite[pp.~11--17]{BJoy}, that 
\begin{align*}
\cal O\left(\frac{ah'}{c}; q_1\right)&=4\sin^2\left(\frac{\pi ah'}{c} \right)\frac{\eta(2z_1)}{\eta(z_1^2)}\sum_{n\in\bb Z}\frac{(-1)^nq_1^{n^2+n}}{1-2q_1^n\cos\left(\frac{2\pi ah'}{c} \right)+q^{2n}_1 }\\&=\frac{\eta(2z_1)}{\eta(z_1^2)}\left( 1+8\sin^2\left(\frac{\pi ah'}{c} \right)\sum_{n\ge1}\frac{(-1)^nq_1^{n^2+n}}{1-2q_1^n\cos\left(\frac{2\pi ah'}{c} \right)+q^{2n}_1 }\right)  
\\&=\ol P(q_1)\left( 1+8\sin^2\left(\frac{\pi ah'}{c} \right)\sum_{n\ge1}\frac{(-1)^nq_1^{n^2+n}}{1-2q_1^n\cos\left(\frac{2\pi ah'}{c} \right)+q^{2n}_1 }\right),
\end{align*} where we set $ q_1=e^{2\pi iz_1}. $
We can rewrite this as 
\[\cal O\left(\frac{ah'}{c}; q_1\right)=1+\sum_{r\ge1}a_1(r)\cdot e^{\frac{2\pi im_rh'}{k}}\cdot e^{-\frac{2\pi r}{kz}},\]
with $ m_r\in  \bb Z $ and the coefficients $ a_1(r) $ being independent of $ k $ and $ h. $ 
Now the sum coming from $ r\ge1 $ will go, as we have seen in the case of $ S_2, $ into an error term of the form $ O(n^{\varepsilon}), $ hence 
\begin{equation*}\label{sum1}
\sum_1= i\tan \left(\frac{\pi a}{c} \right) \sum_{\substack{h,k\\2|k,~c|k}}\frac{\omega^2_{h,k}}{\omega_{h,k/2}}(-1)^{k_1+1}\cot\left(\frac{\pi ah'}{c} \right)e^{-\frac{2\pi ia^2h'k_1}{c}-\frac{2\pi inh}{k}}\int_{-\vartheta'_{h,k}}^{\vartheta''_{h,k}} z^{-\frac12}\cdot e^{\frac{2\pi nz}{k}} d\Phi + O(n^{\varepsilon}).
\end{equation*} As for the sum coming from the constant term, let us denote it simply by $ S, $ 
on splitting the path of integration exactly as in the case of $ S_1 $ and working out the estimates in a similar manner, we obtain
\begin{equation*}
S=i\sum_{c|k,~2|k}A_{a,c,k}(-n,0)\int_{-\frac1{kN}}^{\frac1{kN}}z^{-\frac12}\cdot e^{\frac{2\pi nz }{k} }d\Phi+ O(n^{\varepsilon}).
\end{equation*}
By applying part (iii) of Lemma \ref{estimateKloosterman} and arguing as in the case of $ S_{21} $ (except that now $ m_r=0 $), we get  
\begin{align*}
\left| i\sum_{c|k,~2|k}A_{a,c,k}(-n,0)\int_{-\frac1{kN}}^{\frac1{kN}}z^{-\frac12}\cdot e^{\frac{2\pi nz }{k} }d\Phi\right| &\ll \sum_k k^{\frac12+\varepsilon}\cdot(24n+1,k)^{\frac12}\cdot \frac1{k(N+1)}n^{\frac12}k^{-\frac12}\\&\ll \sum_k k^{-1+\varepsilon}\cdot(24n+1,k)^{\frac12} 
\ll \sum_{\substack{d|24n+1\\d\le N}}d^{-\frac12+\varepsilon}\int_{1}^{N/d}x^{-1+\varepsilon}dx\\&\ll \sum_{\substack{d|24n+1\\d\le N}}d^{-\frac12}\cdot d^{\varepsilon}\cdot\left(\frac{N}{d} \right)^{\varepsilon}  \ll n^{\varepsilon},
\end{align*} proving the claim.
\par We next deal with $ \sum_3 $ and $ \sum_4. $ The reader interested in writing down the computations explicitly will see that the two sums can be expressed as
\[\cal O\left(ah',\frac{\ell c}{c_1},c; q_1\right)=\sum_{r\ge0} a_3(r)\cdot e^{\frac{2\pi i m_rh'}{k}}\cdot e^{-\frac{\pi r}{kc_1^2z}} \]
and 
\[{\cal V}\left(\frac{ah'}{c};q_1 \right)=\sum_{r\ge0} a_4(r)\cdot e^{\frac{2\pi in_rh' }{k}}\cdot e^{-\frac{(2r+1)\pi}{4kz}}, \]
where $ m_r,n_r\in \bb Z $ and the coefficients $ a_3(r) $ and $ a_4(r) $ are independent of $ k $ and $ h. $ Since $ r\ge 0 $, it is obvious that both sums will be of order $ O(n^{\varepsilon}), $ the argument being the same as for $ S_2. $

\subsection{Estimates for the sums $ \sum_7 $ and $ \sum_8 $}
The estimation of the remaining sums $ \sum_7 $ and $ \sum_8 $ is not difficult and is inspired by Bringmann \cite[p. 3497]{B}. Let us, however, elaborate a bit more here. Again, we split the path of integration as in \eqref{splitintegralRademacher}. The resulting sums can each be bounded on the various intervals of integration by 
\begin{equation}
\left( \sum_{k}k^{-1}\right)   \left( \primesum\sum_{h}1 \right) \cdot \sum_{\nu=0}^{k-1} k^{-1}\cdot N^{-1}\cdot z^{\frac12}\cdot I_{a,c,k,v}(z)\ll N^{-1}\cdot n^{\frac12}\cdot k^{-\frac12}\cdot g_{a,c,k,\nu}\ll k^{\varepsilon}\ll n^{\varepsilon},
\end{equation} 
for any $ \varepsilon>0. $ Here we have used, in turn, a trivial bound for the Kloosterman sums appearing in front of the integrals from $ \sum_7 $ and $ \sum_8, $ Lemma \ref{LemmawithMordell}, and the easy estimate 
\[\sum_{1\le \nu\le k}g_{a,c,k,\nu}\ll \sum_{1\le \nu \le4ck}\frac1{\nu}\ll k^{\varepsilon}.\]
By this we conclude the rather lengthy proof of Theorem \ref{MainThm}.
\end{proof}
\begin{proof}[Proof of Corollary \ref{Corollary2}]
Let us first assume $ c>2. $ On combining Theorem \ref{MainThm} and identity \eqref{orthogonality}, we obtain 
\begin{align}\label{Nacn}
\overline N(a,c,n)&=\frac1c\sum_{j=1}^{c-1}\zeta_c^{-aj}\Bigg(  i\sqrt{\frac 2n}\sum_{c'| k,~2\nmid k}\frac{B_{j',c',k}(-n,0)}{\sqrt k}\cdot \sinh\left( \frac{\pi\sqrt n}{k}\right)  \nonumber\\
&\phantom{=~} +2\sqrt{\frac 2n}\sum_{\substack{c'\nmid k,~2\nmid k,~\widetilde c\ne4\\r\ge0,~ \delta_{c',k,r}>0}}\frac{D_{j',c',k}(-n,m_{j',c',k,r})}{\sqrt k} \cdot\sinh \left( \frac{4\pi\sqrt{\delta_{c',k,r}n}}{k} \right)\Bigg)\nonumber\\
&\phantom{=~}+\frac{\overline p(n)}c  +O_c(n^{\varepsilon}),
\end{align} 
where $ c'=\frac c{(c,j)}, $ $ j'=\frac j{(c,j)},$  $\widetilde c=\frac{c'}{(c',k)} $    and $ \varepsilon>0 $ is arbitrary. As $ n\to\infty, $ we know that  \[\overline p(n)\sim \frac1{8n}e^{\pi\sqrt n}.\] Since $ c'\ge 2 $ (as $ j\le c-1), $ summation of the $ B_{j',c',k} $ terms in \eqref{Nacn} can only start from $ k=3, $ meaning that the asymptotic contribution of these sums is (at most) of order $\sinh\left( \frac{\pi\sqrt n}{3}\right),  $ thus dominated by $ \ol p(n). $   $  $ \par We claim that the same is true for the contribution coming from the $ D_{j',c',k} $ sums.
For this, note that, directly from the definition \eqref{delta}, it follows that $ \delta_{c,k,r}\le \frac1{16}, $ therefore \[\sinh \left( \frac{4\pi\sqrt{\delta_{c,k,r}n}}{k} \right)\le \sinh\left( \frac{\pi\sqrt n}{k}\right).\]
If summation of the $ D_{j',c',k} $ terms in \eqref{Nacn} starts from $ k=3 $ (note that $ 2\nmid k $), then there is nothing to prove; so assume $ k=1. $ It is an easy exercise to prove that equality above cannot be, in fact, obtained, and that, since $ c_1=c,$ we have   $ \delta_{c,k,r}\le \frac1{16}-\frac1{2c}+\frac1{c^2}=\frac1{16}-\frac{c-2}{2c^2}, $ with $ c\ge3,$ thereby proving the claim. 
\par In case $ c=2, $ we leave it as an exercise, to the interested reader, to prove that the coefficients of $ \cal O(-1;q) $ are of order $ O(n^{\varepsilon}) $ and are thus dominated by $ \ol p(n). $ This can be done by using the transformation behavior described in \cite[Corollary 4.2]{BJoy} and carrying out estimates similar to those from the proof of Theorem \ref{MainThm}.
\end{proof}
\section{A few inequalities}\label{Ineqs}
In this section we prove the inequalities stated in Theorems \ref{Thm2}--\ref{Thm4}.~We will elaborate more on Theorem \ref{Thm2}, while only sketching the main steps in the proofs of Theorems \ref{Thm3} and \ref{Thm4}, as the ideas are similar. 
\par Before giving the proof of Theorem \ref{Thm2}, we must establish some identities. The following is an easy generalization of \cite[Lemma 3.1]{JZZ}.
\begin{Lem}\label{Lemma:zetarelations10}
If $ a\in\N $ is odd and $ 5\nmid a $, then 
\begin{align*}\label{N1234}
\O(\zeta_{10}^a;q)&=\sum_{n=0}^{\infty}(\overline N(0,10,n)+\overline N(1,10,n)-\overline N(4,10,n)-\overline N(5,10,n))q^n\nonumber\\
&\phantom{=~}+(\zeta_{10}^{2a}-\zeta_{10}^{3a})\sum_{n=0}^{\infty}(\overline N(1,10,n)+\overline N(2,10,n)-\overline N(3,10,n)-\overline N(4,10,n))q^n.
\end{align*}
\end{Lem}
\begin{proof}
Plugging $ u=\zeta^a_{10} $ into \eqref{overpartsrank} gives
\begin{equation}\label{zeta10a}
\O(\zeta_{10}^a;q)=\sum_{n=0}^{\infty}\overline{N}(m,n)\zeta_{10}^{am}q^n=\frac{(-q)_{\infty}}{(q)_{\infty}}\sum_{n=-\infty}^{\infty}\frac{(1-\zeta_{10}^a)(1-\zeta_{10}^{-a})(-1)^nq^{n^2+n}}{(1-\zeta_{10}^aq^n)(1-\zeta_{10}^{-a}q^n)}.  
\end{equation}
Using the fact that $ \overline N(a,m,n)=\overline N(m-a,m,n), $ which can be easily deduced from $ \ol N(m,n)=\ol N(-m,n) $ (see, e.g., \cite[Proposition 1.1]{Lovejoy}), and noting that $ \zeta_{10}^{5a}=-1 $ and $ 1-\zeta_{10}^a+\zeta^{2a}_{10}-\zeta_{10}^{3a}+\zeta_{10}^{4a}=0  $ for $ 5\nmid a $ odd,  we can rewrite \eqref{zeta10a} as
\begin{align*}
\O(\zeta_{10}^a;q)&=\sum_{n=0}^{\infty}\sum_{\ell=0}^9\overline N(\ell ,10,n)\zeta_{10}^{\ell a}q^n\\
&=\sum_{n=0}^{\infty}(\overline N(0,10,n)+(\zeta_{10}^a-\zeta_{10}^{4a})\overline N(1,10,n)+(\zeta_{10}^{2a}-\zeta_{10}^{3a})\overline N(2,10,n)\\
&\phantom{=~}+(\zeta_{10}^{3a}-\zeta_{10}^{2a})\overline N(3,10,n)+(\zeta_{10}^{4a}-\zeta_{10}^{a})\overline N(4,10,n)-\overline N(5,10,n))q^n\\
&=\sum_{n=0}^{\infty}(\overline N(0,10,n)+(1+\zeta_{10}^{2a}-\zeta_{10}^{3a})\overline N(1,10,n)+(\zeta_{10}^{2a}-\zeta_{10}^{3a})\overline N(2,10,n)\\
&\phantom{=~}+(\zeta_{10}^{3a}-\zeta_{10}^{2a})\overline N(3,10,n)-(1+\zeta_{10}^{2a}-\zeta_{10}^{3a})\overline N(4,10,n)-\overline N(5,10,n))q^n\\
&=\sum_{n=0}^{\infty}(\overline N(0,10,n)+\overline N(1,10,n)+\overline N(4,10,n)-\overline N(5,10,n)q^n\\
&\phantom{=~}+(\zeta_{10}^{2a}-\zeta_{10}^{3a})\sum_{n=0}^{\infty}(\overline N(1,10,n)+\overline N(2,10,n)-\overline N(3,10,n)-\overline N(4,10,n))q^n,
\end{align*}
which concludes the proof.
\end{proof}
In a similar fashion, we have the following result. For a proof of the case $ a=1, $ see \cite[Lemma 2.1]{JZZ}. 
\begin{Lem}\label{Lemma:zetarelations6}
If $ a\in\N $ is odd and $ 3\nmid a, $ then 
\begin{equation*}
\O(\zeta_{6}^a;q)=\sum_{n=0}^{\infty}(\overline N(0,6,n)+\overline N(1,6,n)-\overline N(2,6,n)-\overline N(3,6,n))q^n.
\end{equation*}
\end{Lem}
\begin{proof}[Proof of Theorem \ref{Thm2}]
Setting $ a=1 $ and $ a=3 $ in Lemma \ref{Lemma:zetarelations10}, we obtain 
\begin{align}\label{rootunitya=1}
\O(\zeta_{10};q)
&=\sum_{n=0}^{\infty}(\overline N(0,10,n)+\overline N(1,10,n)-\overline N(4,10,n)-\overline N(5,10,n))q^n\nonumber\\
&\phantom{=~}+(\zeta_{10}^{2}-\zeta_{10}^{3})\sum_{n=0}^{\infty}(\overline N(1,10,n)+\overline N(2,10,n)-\overline N(3,10,n)-\overline N(4,10,n))q^n,
\end{align} and 
\begin{align}\label{rootunitya=3}
\O(\zeta_{10}^3;q)
&=\sum_{n=0}^{\infty}(\overline N(0,10,n)+\overline N(1,10,n)-\overline N(4,10,n)-\overline N(5,10,n))q^n\nonumber\\
&\phantom{=~}+(\zeta_{10}^{4}-\zeta_{10})\sum_{n=0}^{\infty}(\overline N(1,10,n)+\overline N(2,10,n)-\overline N(3,10,n)-\overline N(4,10,n))q^n.
\end{align}
Subtracting \eqref{rootunitya=3} from \eqref{rootunitya=1} yields
\[\sum_{n=0}^{\infty}(\overline N(1,10,n)+\overline N(2,10,n)-\overline N(3,10,n)-\overline N(4,10,n))q^n=\frac{\O(\zeta_{10};q)-\O(\zeta_{10}^3;q)}{\zeta_{10}+\zeta_{10}^2-\zeta_{10}^3-\zeta_{10}^4}=\frac{\O(\zeta_{10};q)-\O(\zeta_{10}^3;q)}{1+4\cos\left( \frac{2\pi}{5}\right) },\]
thus proving \eqref{eq:1234} is equivalent to showing that, for $ n\ge0, $
\[A\left(\frac1{10};n \right)\ge A\left(\frac3{10};n\right) . \] 
For $ a=1,$ $c=10 $ we have $ m_{1,10,1,0}=0 $ and $ \delta_{c,k,r}>0$ if and only if $ r=0, $ in which case $ \delta_{c,k,r}=\frac9{400}, $ hence
\begin{equation}\label{A1/10}
A\left(\frac1{10};n \right)= 2\sqrt{\frac 2n}\sum_{\substack{1\le k\le \sqrt n\\ k\equiv1,9\pmod*{10}}}\frac{D_{a,c,k}(-n,m_{1,10,k,0})}{\sqrt k}\cdot \sinh \left( \frac{3\pi\sqrt{n}}{5k} \right)+O_c(n^{\varepsilon}),
\end{equation} 
whereas, for $ a=3$ and $c=10, $ we have $ \delta_{c,k,r}>0$ if and only if $r=0, $ in which case $ \delta_{c,k,r}=\frac9{400}, $ thus 
\begin{equation}\label{A3/10}
A\left(\frac3{10};n \right)= 2\sqrt{\frac 2n}\sum_{\substack{1\le k\le \sqrt n\\ k\equiv3,7\pmod*{10}}}\frac{D_{a,c,k}(-n,m_{3,10,k,0})}{\sqrt k}\cdot \sinh \left( \frac{3\pi\sqrt{n}}{5k} \right)+O_c(n^{\varepsilon}).
\end{equation}
We further compute $$ D_{1,10,1}(-n,0)=\frac1{\sqrt2}\tan\left( \frac{\pi}{10}\right),  $$ and so the term corresponding to $ k=1 $ in the sum from \eqref{A1/10} is given by $$\frac2{\sqrt n}\tan\left( \frac{\pi}{10}\right)\sinh \left( \frac{3\pi\sqrt{n}}{5} \right). $$
Using a trivial bound for the Kloosterman sum from \eqref{A3/10} and taking into account the contributions coming from the various error terms involved, estimates which we make explicit at the end of this section, we see that this term is dominant for $ n\ge1030, $ hence 
\[A\left(\frac1{10};n \right)\ge A\left(\frac3{10};n\right)  \]
for $ n\ge1030. $ In Mathematica we see that the inequality is true for $ n<1030 $ as well.
\par To prove \eqref{eq:0325}, we set
$ a=1 $ and $ a=3 $ in Lemma \ref{Lemma:zetarelations10} and obtain
\begin{align}\label{rootunitya=1rewrite}
\O(\zeta_{10};q)
&=\sum_{n=0}^{\infty}(\overline N(0,10,n)+\overline N(3,10,n)-\overline N(2,10,n)-\overline N(5,10,n))q^n\nonumber\\
&\phantom{=~}+(1+\zeta_{10}^{2}-\zeta_{10}^{3})\sum_{n=0}^{\infty}(\overline N(1,10,n)+\overline N(2,10,n)-\overline N(3,10,n)-\overline N(4,10,n))q^n,
\end{align} and 
\begin{align}\label{rootunitya=3rewrite}
\O(\zeta_{10}^3;q)
&=\sum_{n=0}^{\infty}(\overline N(0,10,n)+\overline N(3,10,n)-\overline N(2,10,n)-\overline N(5,10,n))q^n\nonumber\\
&\phantom{=~}+(1-\zeta_{10}+\zeta_{10}^{4})\sum_{n=0}^{\infty}(\overline N(1,10,n)+\overline N(2,10,n)-\overline N(3,10,n)-\overline N(4,10,n))q^n.
\end{align}
Combining \eqref{rootunitya=1rewrite} and \eqref{rootunitya=3rewrite} and setting $ \alpha=\frac{1+\zeta_{10}^{2}-\zeta_{10}^{3}}{1-\zeta_{10}+\zeta_{10}^{4}}, $ we obtain
\begin{equation*}
 \O(\zeta_{10};q)-\alpha\cdot\O(\zeta_{10}^3;q) =(1-\alpha) \sum_{n=0}^{\infty}(\overline N(0,10,n)+\overline N(3,10,n)-\overline N(2,10,n)-\overline N(5,10,n))q^n,
\end{equation*} 
hence, as it is easy to see that $ \alpha=-( 1+2\cos(\pi/5)) , $  proving the claim amounts to showing
\begin{equation*}
A\left(\frac1{10};n \right)+\left(1+2\cos\frac{\pi}{5} \right)A\left( \frac3{10};n\right)\ge0 
\end{equation*}
for all $ n\ge0, $ which follows from the estimates used for proving \eqref{eq:1234}. 
The proof of \eqref{eq:0145} follows simply on adding the inequalities \eqref{eq:1234} and \eqref{eq:0325}.
\end{proof}
We can also sketch now the proofs of Theorems \ref{Thm3} and \ref{Thm4}.
\begin{proof}[Proof of Theorem \ref{Thm3} (Sketch)]
Reasoning along the same lines, on setting $ a=1 $ in Lemma \ref{Lemma:zetarelations6} and recalling \eqref{OA}, the claim is equivalent to proving \[A\left(\frac1{6};n \right)\ge0 \] for $ n\ge0. $ It is easy to see that, for $ a=1$ and  $c=6, $ we have $ m_{1,6,1,0}=0 $ and $ \delta_{c,k,r}>0$ if and only if $ r=0, $ in which case $ \delta_{c,k,r}=\frac1{144}, $ thus the dominant term of 
\begin{equation*}
A\left(\frac1{6};n \right)= 2\sqrt{\frac 2n}\sum_{\substack{1\le k\le \sqrt n\\ k\equiv1,5\pmod*{6}}}\frac{D_{1,6,k}(-n,m_{1,6,k,0})}{\sqrt k}\cdot \sinh \left( \frac{3\pi\sqrt{n}}{5k} \right)
\end{equation*}
is given by \[\frac2{\sqrt n}\tan\left( \frac{\pi}{6}\right)\sinh \left( \frac{\pi\sqrt{n}}{3} \right). \]
By working out similar bounds as in the proof of \eqref{eq:1234} and checking numerically for the small values of $ n, $ the proof of \eqref{eq:0123} is concluded. 
\par The inequalities \eqref{eq:0312}--\eqref{eq:0312primeprime} are equivalent to those from \eqref{SolvedWei11}--\eqref{SolvedWei33}. The proof relies on the identity 
\begin{align*}
\O(\zeta_6^2;q)&=\sum_{n=0}^{\infty}(\ol N(0,6,n)-\ol N(1,6,n)-\ol N(2,6,n)+\ol N(3,6,n))q^n\\&=\sum_{n=0}^{\infty}(\ol N(0,6,n)-\ol N(1,6,n)-\ol N(4,6,n)+\ol N(3,6,n))q^n\\&=\sum_{n=0}^{\infty}(\ol N(0,3,n)-\ol N(1,3,n))q^n
\end{align*} and details are left to the interested reader. The fact that $ \ol N(1,3,n)=\ol N(2,3,n) $ follows easily from adding the identities $ \ol N(1,6,n)=\ol N(5,6,n) $ and $ \ol N(2,6,n)=\ol N(4,6,n). $ 
\end{proof}
\begin{proof}[Proof of Theorem \ref{Thm4} (Sketch)] 

By using either \cite[Lemma 5.1]{WZ} (on identifying the notation $ \ol R(u;q)=\O (u;q) $) or identity \eqref{orthogonality} (which, in combination with \eqref{overpartsrank}, amounts to the same result), we have
\begin{align}
&\label{Mao1}\sum_{n=0}^{\infty}\ol N(0,6,n)q^n=\frac16(\O (1;q)+2\O (\zeta_6;q)+2\O(\zeta_6^2;q)+\O(\zeta_6^3;q)),\\
&\label{Mao2}\sum_{n=0}^{\infty}\ol N(1,6,n)q^n=\frac16(\O (1;q)+\phantom{2}\O (\zeta_6;q)-\phantom{2}\O(\zeta_6^2;q)-\O(\zeta_6^3;q)),\\
&\label{Mao3}\sum_{n=0}^{\infty}\ol N(2,6,n)q^n=\frac16(\O (1;q)-\phantom{2}\O (\zeta_6;q)-\phantom{2}\O(\zeta_6^2;q)+\O(\zeta_6^3;q)),\\
&\label{Mao4}\sum_{n=0}^{\infty}\ol N(3,6,n)q^n=\frac16(\O (1;q)-2\O (\zeta_6;q)+2\O(\zeta_6^2;q)-\O(\zeta_6^3;q)).
\end{align}
In light of Remark \ref{remarkmod6}, to prove the inequalities \eqref{SolvedWei1}--\eqref{SolvedWei3} it suffices to show that, for $ n\ge11, $ \[\ol N(1,6,n)\ge \ol N(2,6,n),\]\[\ol N(0,6,3n)\ge \ol N(1,6,3n),
\quad\ol N(0,6,3n+1)\ge \ol N(1,6,3n+1),\]
\[ \ol N(0,6,3n+2)\le \ol N(1,6,3n+2).\]
Therefore, on combining \eqref{Mao2} and \eqref{Mao3}, the first inequality above is equivalent to  
\begin{equation}\label{final16two}
A\left(\frac16;n \right)\ge 0,
\end{equation} whereas, for  $ i=0,1,$  the second and third are equivalent, on combining \eqref{Mao1} and \eqref{Mao2}, to 
\begin{equation}\label{final16one}
A\left(\frac16;3n+i \right)+ 3A\left( \frac13;3n+i\right)\ge0\quad\text{and}\quad A\left(\frac16;3n+2 \right)+ 3A\left( \frac13;3n+2\right)\le0.
\end{equation}
\par Again, the attentive reader might wonder what happens with the term $ \O (-1;q) $ (coming from the case $ j=c/2 $ in \eqref{orthogonality}), to which Theorem \ref{Transformations} does not apply, as its statement is formulated under the assumption $ c>2. $ However, while working out the transformations found by Bringmann and Lovejoy in this case, see \cite[Corollary 4.2]{BJoy}, and doing the same estimates as in the proof of Theorem \ref{MainThm}, one can easily infer that the sums involved are of order $ O(n^{\varepsilon}). $ Therefore, as $ n $ grows large, we only need to prove \eqref{final16two} and \eqref{final16one}, which follow immediately from Theorem \ref{MainThm}. Again, explicit bounds can be provided just as described in the next subsection, and a numerical check for the small values of $ n $ concludes the proof.
\end{proof} 

\subsection{Some explicit computations}
As we have mentioned earlier, we will now fill in the missing details from the proof of \eqref{eq:1234}, by explaining how to bound the different sums and error terms appearing in \eqref{A1/10} and \eqref{A3/10}. The same arguments apply for all the other inequalities. We have already seen that
\begin{equation*}
A\left(\frac1{10};n \right)= 2\sqrt{\frac 2n}\sum_{\substack{1\le k\le \sqrt n\\ k\equiv1,9\pmod*{10}}}\frac{D_{a,c,k}(-n,m_{1,10,k,0})}{\sqrt k}\cdot \sinh \left( \frac{3\pi\sqrt{n}}{5k} \right),
\end{equation*} 
and that the term corresponding to $ k=1 $ in \eqref{A1/10} equals 
\begin{equation}\label{mainterm}
\frac2{\sqrt n}\tan\left( \frac{\pi}{10}\right)\sinh \left( \frac{3\pi\sqrt{n}}{5} \right). 
\end{equation}
By using a trivial bound for the Kloosterman sums involved, the remaining terms can be estimated against 
\begin{equation}\label{maintermA1/10}
\frac4{\sqrt n}\sum_{2\le k\le\frac{N-1}{10}}k^{\frac12}\cdot \sinh \left( \frac{3\pi\sqrt{n}}{5(10k+1)} \right)+\frac4{\sqrt n}\sum_{1\le k\le\frac{N-9}{10}}k^{\frac12}\cdot \sinh \left( \frac{3\pi\sqrt{n}}{5(10k+9)} \right),
\end{equation}
and the contribution coming from $ \U\left(h',\frac{\l}{10},10;q_1 \right)  $ is seen to be less than
\begin{equation}\label{a5(r)}
\sqrt2\cdot e^{2\pi}\sum_{r=1}^{\infty}|a_5(r)|\cdot e^{-\frac{\pi r}{50}}\sum_{\substack{1\le k\le N\\k\equiv1,9\pmod*{10}}} k^{-\frac12}+\sqrt2\cdot e^{2\pi}\sum_{r=1}^{\infty}|b_5(r)|\cdot e^{-\frac{\pi r}{50}}\sum_{\substack{1\le k\le N\\k\equiv1,9\pmod*{10}}} k^{-\frac12}.
\end{equation}
Making the path of integration symmetric in \eqref{sigma2} introduces an error that can be estimated against 
\begin{equation}\label{integrationpathsym}
2\cdot e^{2\pi +\frac{\pi}{8}}\cdot n^{-\frac12}\sum_{\substack{1\le k\le N\\k\equiv1,9\pmod*{10}}} k^{\frac12},
\end{equation}
while integrating along the smaller arc of $ \Gamma $ gives an error not bigger than
\begin{equation}\label{smallerarc}
8\pi \cdot e^{2\pi+\frac{\pi}{16}}\cdot n^{-\frac34}\sum_{\substack{1\le k\le N\\k\equiv1,9\pmod*{10}}} k. 
\end{equation}
The sums $ \sum_2,\sum_4 $ and $ \sum_6 $ do not contribute in the case $ c=10, $ whereas $ \sum_1,\sum_3 $ can be treated simultaneously. The contribution coming from $ \O\left(\frac{h'}{10};q_1 \right)$ can be estimated against 
\begin{equation}\label{Ohprime/10}
\frac{2\cdot e^{2\pi}}{\sqrt{10}}\sum_{1\le k\le \frac N{10}}k^{-\frac12}+\frac{2\cdot e^{2\pi}}{\sqrt{10}}\sum_{r=1}^{\infty}|a_1(r)|\cdot e^{-\pi r}\sum_{1\le k\le \frac N{10}}k^{-\frac12}, 
\end{equation}
and that coming from $ \O\left(h',\frac{\l}{2},10;q_1 \right)  $
against 
\begin{equation}\label{Ohprime/10ellover2}
2\cdot e^{2\pi}\sum_{r=1}^{\infty}|a_3(r)|\cdot e^{-\frac{\pi r}{50}}\sum_{\substack{1\le k\le N\\k\equiv1,9\pmod*{10}}} k^{-\frac12}.
\end{equation} 
Using the bound \eqref{asympoverpartitions} for $ |a_3(r)|,  |a_5(r)| $ and $ |b_5(r)|, $ we get $ \sum_{r=1}^{\infty}|a_3(r)|\cdot e^{-\frac{\pi r}{50}}<1.17944 $ and $\sum_{r=1}^{\infty}|a_3(r)|\cdot e^{-\frac{\pi r}{50}}<4.01014\cdot 10^{19},  $ and similarly for $ a_5(r) $ and $ b_5(r). $
Finally, the estimates in Lemma \ref{LemmawithMordell} can be made explicit so as to give
\begin{equation}\label{sum7}
\sum_7\le \frac{2e^{2\pi}\sqrt{\pi}}5\cdot   \sum_{2\nmid k}k^{-\frac32}\sum_{\nu=1}^{k} \left(\min \left\lbrace \left| \frac{\nu}{k}-\frac1{4k}+\frac 1{10}\right| , \left| \frac{\nu}{k}-\frac1{4k}-\frac 1{10} \right| \right\rbrace  \right) ^{-1}
\end{equation}
and 
\begin{equation}\label{sum8}
\sum_8\le \frac{2e^{2\pi}\sqrt{2\pi}}5
\cdot \Bigg( \sum_{2|k,~5\nmid k }k^{-\frac32}\sum_{\nu=1}^{k} \left(\min \left\lbrace \left| \frac{\nu}{k}+\frac 1{10}\right| , \left| \frac{\nu}{k}-\frac 1{10}\right| \right\rbrace \right) ^{-1}+\frac1{10\sqrt{10}}\sum_{1\le k\le \frac N{10}}k^{-\frac12} \Bigg). 
\end{equation}
For $ a=3 $ and $ c=10, $ we proceed just like in \eqref{maintermA1/10} to get 
\[\frac4{\sqrt n}\sum_{1\le k\le\frac{N-3}{10}}k^{\frac12}\cdot \sinh \left( \frac{3\pi\sqrt{n}}{5(10k+3)} \right)+\frac4{\sqrt n}\sum_{1\le k\le\frac{N-7}{10}}k^{\frac12}\cdot \sinh \left( \frac{3\pi\sqrt{n}}{5(10k+7)} \right)\]
as an overall bound for the main contribution in \eqref{A3/10} and we use the same estimates from \eqref{a5(r)}--\eqref{sum8} on changing whatever necessary, e.g., the sums will now run over $ k\equiv3\pmod*{10}$ and $k\equiv 7\pmod*{10}. $ Putting all estimates together, we see that the term in \eqref{mainterm} is dominant for $ n\ge1030. $ The inequality \eqref{eq:1234} can be checked numerically in Mathematica to hold true also for $ n<1030. $
\section*{Acknowledgments}
The author is grateful to Kathrin Bringmann for suggesting this project and for many useful discussions, and to Chris Jennings-Shaffer for permanent advice, as well as for the patience of carefully reading an earlier version of the manuscript and double-checking some computations. The author would also like to kindly thank the anonymous referee for the very helpful suggestions made on improving the exposition of this paper.~The work was supported by the European Research Council under the European Union's Seventh Framework Programme (FP/2007--2013)/ERC Grant agreement n. 335220 --- AQSER.


\begin{thebibliography}{99}
\bibitem{AIS} E. Alwaise, E. Iannuzzi and H. Swisher, A proof of some conjectures of Mao on partition rank inequalities, \textit{Ramanujan J.} \textbf{43} (2017), no. 3, 633--648.
\bibitem{AndrewsThesis} G. E. Andrews, On the theorems of Watson and Dragonette for Ramanujan's mock theta functions, \textit{Amer. J. Math.} \textbf{88} no. 2 (1966), 454--490.
\bibitem{Apo} T. M. Apostol, \textit{Introduction to analytic number theory}, Undergraduate Texts in Mathematics, New York-Heidelberg: Springer-Verlag, 1976.
\bibitem{Apostol} T. M. Aposotol, \textit{Modular functions and Dirichlet series in number theory}, Second edition, Graduate Texts in Mathematics, 41. Springer-Verlag, New York, 1990.
\bibitem{ASD} A. O. L. Atkin and P. Swinnerton-Dyer, Some properties of partitions, \textit{Proc. London Math. Soc. (3)} \textbf{4} (1954), 84--106.
\bibitem{B} K. Bringmann, Asymptotics for rank partition functions, \textit{Trans. Amer. Math. Soc.} \textbf{361} (2009), no. 7, 3483--3500.
\bibitem{BJoy} K. Bringmann and J. Lovejoy, Dyson's rank, overpartitions, and weak Maass Forms, \textit{Int. Math. Res. Not. IMRN} (2007), no. 19, Art. ID rnm063, 34 pp.	
\bibitem{BOno} K. Bringmann and K. Ono, Dyson's ranks and Maass forms, \textit{Ann. of Math. (2)} \textbf{171} (2010), no. 1, 419--449. 
\bibitem{CL} S. Corteel and J. Lovejoy, Overpartitions, \textit{Trans. Amer. Math. Soc.} \textbf{356} (2004), no. 4, 1623--1635.
\bibitem{Dragon} L. A. Dragonette, Some asymptotic formul\ae~ for the mock theta series of Ramanujan, 
\textit{Trans. Amer. Math. Soc.} \textbf{72} (1952), 474--500.  
\bibitem{Dewar} M. Dewar, The nonholomorphic parts of certain weak Maass forms, \textit{J. Number Theory} \textbf{130} (2010), no. 3, 559--573. 
\bibitem{Dyson} F. J. Dyson, Some guesses in the theory of partitions, \textit{Eureka} (1944), no. 8, 10--15. 
\bibitem{HR} G. H. Hardy and S. Ramanujan, Asymptotic formul\ae~ for the distribution of integers of various types,
\textit{Proc. London Math. Soc. Series (2)}  \textbf{16} (1918), 112--132.
\bibitem{CJS} C. Jennings-Shaffer, Overpartition rank differences modulo 7 by Maass forms, \textit{J. Number Theory} \textbf{163} (2016), 331--358.
\bibitem{CJSREIH} C. Jennings-Shaffer and D. Reihill, Asymptotic formulas related to the $ M_2 $-rank of partitions without repeated odd parts, available as preprint at \url{https://arxiv.org/abs/1805.11319}.
\bibitem{JZZ} K. Q. Ji, H. W. J. Zhang and A. X. H. Zhao, Ranks of overpartitions modulo $ 6 $ and $ 10, $ \textit{J. Numbery Theory} \textbf{184} (2018), 235--269.
\bibitem{Lovejoy} J. Lovejoy, Rank and conjugation for the Frobenius representation of an overpartition, \textit{Ann. Comb.} \textbf{9} (2005), no. 3, 321--334.
\bibitem{LO} J. Lovejoy and R. Osburn, Rank differences for overpartitions, \textit{Q. J. Math.} \textbf{59} (2) (2008), 257--273.
\bibitem{Males} J. Males, Asymptotic equidistribution and convexity for partition ranks, available as preprint at \url{https://arxiv.org/abs/1903.05857}.
\bibitem{Mao} R. Mao, Ranks of partitions modulo $ 10, $ \textit{J. Number Theory} \textbf{133} (2013) 3678--3702.
\bibitem{Mao2} R. Mao, The $M_2$-rank of partitions without repeated odd parts modulo $ 6 $ and $ 10, $ \textit{Ramanujan J.} \textbf{37} (2015), no. 2, 391--419.

\bibitem{Pak} I. Pak, Partition bijections, a survey, \textit{Ramanujan J.} \textbf{12} (2006), no. 1, 5--75. 
\bibitem{Rademacher} H. Rademacher, The Fourier coefficients of the modular invariant $ J(\tau), $ \textit{American Journal of Mathematics} \textbf{60} (1938), 501--512.
\bibitem{Ram}  S. Ramanujan, Congruence properties of partitions, \textit{Math. Z.} \textbf{9} (1921), 147--153.
University Press, 1927, 355 pp.
\bibitem{Watson} G. N. Watson, \textit{A treatise on the theory of Bessel functions}, Second edition, Cambridge University Press, 1944. 
\bibitem{WZ} B. Wei and H. W. J. Zhang, Generalized Lambert series identities and applications in rank differences, available as preprint at \url{https://arxiv.org/abs/1801.04643}.
\bibitem{Zucker} H. S. Zuckerman, On the coefficients of certain modular forms belonging to subgroups of the
modular group, \textit{Trans. Amer. Math. Soc.} \textbf{45} (1939), no. 2, pp. 298--321.
\end{thebibliography}
\end{document}